\documentclass[reqno,11pt]{amsart}

\usepackage[a4paper,left=23mm,right=23mm,top=30mm,bottom=30mm,marginpar=25mm]{geometry}
\usepackage{amsmath}
\usepackage{amssymb}
\usepackage{amsthm}
\usepackage{amscd}
\usepackage{mathtools}
\usepackage{url}
\usepackage[T1]{fontenc}
\usepackage{color}
\usepackage[final]{graphicx}
\usepackage{enumitem}
\usepackage{esint}
\usepackage{dsfont}
\usepackage{subfigure}
\usepackage{tikz}
\usepackage{cite}
\usepackage[colorlinks=true, linkcolor=blue,  anchorcolor=blue,citecolor=blue]{hyperref}
\usepackage[english]{babel}

\numberwithin{equation}{section}

\newtheoremstyle{thmlemcorr}{10pt}{10pt}{\itshape}{}{\bfseries}{.}{10pt}{{\thmname{#1}\thmnumber{ #2}\thmnote{ (#3)}}}
\newtheoremstyle{thmlemcorr*}{10pt}{10pt}{\itshape}{}{\bfseries}{.}\newline{{\thmname{#1}\thmnumber{ #2}\thmnote{ (#3)}}}
\newtheoremstyle{defi}{10pt}{10pt}{\itshape}{}{\bfseries}{.}{10pt}{{\thmname{#1}\thmnumber{ #2}\thmnote{ (#3)}}}
\newtheoremstyle{remexample}{10pt}{10pt}{}{}{\bfseries}{.}{10pt}{{\thmname{#1}\thmnumber{ #2}\thmnote{ (#3)}}}
\newtheoremstyle{ass}{10pt}{10pt}{}{}{\bfseries}{.}{10pt}{{\thmname{#1}\thmnumber{ A#2}\thmnote{ (#3)}}}

\theoremstyle{thmlemcorr}
\newtheorem{theorem}{Theorem}
\numberwithin{theorem}{section}
\newtheorem{lemma}[theorem]{Lemma}

\newtheorem{proposition}[theorem]{Proposition}

\theoremstyle{thmlemcorr*}
\newtheorem{theorem*}{Theorem}
\newtheorem{lemma*}[theorem]{Lemma}
\newtheorem{corollary*}[theorem]{Corollary}
\newtheorem{proposition*}[theorem]{Proposition}
\newtheorem{problem*}[theorem]{Problem}
\newtheorem{conjecture*}[theorem]{Conjecture}

\theoremstyle{defi}
\newtheorem{definition}[theorem]{Definition}

\theoremstyle{remexample}
\newtheorem{remarkx}[theorem]{Remark}
\newenvironment{remark}
  {\pushQED{\qed}\remarkx}
  {\popQED\endremarkx}

\theoremstyle{ass}

\newcommand{\N}{\mathbb{N}}
\newcommand{\R}{\mathbb{R}}

\newcommand{\Z}{\mathbb{Z}}

\renewcommand{\S}{\mathbb{S}}

\newcommand{\Ecal}{\mathcal{E}}
\newcommand{\Fcal}{\mathcal{F}}

\newcommand{\Hcal}{\mathcal{H}}

\newcommand{\Mcal}{\mathcal{M}}

\renewcommand{\epsilon}{\varepsilon}
\newcommand{\dd}{\,\mathrm{d}}
\DeclareMathOperator*{\essinf}{ess\,inf}
\DeclarePairedDelimiter{\abs}{\lvert}{\rvert}
\DeclarePairedDelimiter{\norm}{\lVert}{\rVert}

\DeclarePairedDelimiter{\inner}{\langle}{\rangle}

\newcommand{\1}{\mathds{1}}
\DeclareMathOperator{\supp}{supp}
\DeclareMathOperator{\Dom}{Dom}

\usepackage{trimclip}

\DeclareRobustCommand{\wtwoscale}{%
  \clippedrightharpoonup\mathrel{\mspace{-15mu}}\rightharpoonup
}
\makeatletter
\newcommand{\clippedrightharpoonup}{%
  \mathrel{\mathpalette\clipped@rightharpoonup\relax}%
}
\newcommand{\clipped@rightharpoonup}[2]{%
  \sbox\z@{$\m@th#1\mspace{6mu}$}%
  \clipbox{0pt 0pt {\dimexpr\width-.5\wd\z@} 0pt}{$\m@th#1\rightharpoonup$}%
  \smash{\clipbox{{\wd\z@} 0pt 0pt {-\width}}{$\m@th#1\rightharpoonup$}}%
}
\makeatother

\DeclareRobustCommand{\twoscale}{%
  \clippedto\mathrel{\mspace{-15mu}}\to
}
\makeatletter
\newcommand{\clippedto}{%
  \mathrel{\mathpalette\clipped@to\relax}%
}
\newcommand{\clipped@to}[2]{%
  \sbox\z@{$\m@th#1\mspace{6mu}$}%
  \clipbox{0pt 0pt {\dimexpr\width-.5\wd\z@} 0pt}{$\m@th#1\to$}%
  \smash{\clipbox{{\wd\z@} 0pt 0pt {-\width}}{$\m@th#1\to$}}%
}
\makeatother

\newcommand{\xtwoscale}{\overset{x}{\twoscale}}
\newcommand{\wxtwoscale}{\overset{x}{\wtwoscale}}

\newcommand{\Om}{\Omega}
\newcommand{\veps}{\varepsilon}

\newcommand{\p}{\partial}
\newcommand{\dl}{\,\mathrm{d}}
\newcommand{\seq}[1]{(#1_k)_{k}}
\newcommand{\sequence}[2]{(#1_#2)_{#2}}

\title{Homogenization of nonlocal exchange energies in micromagnetics}
\AtEndDocument{
    \medskip\hrule	\medskip
	\small{
		\noindent
		{\sc R. Giorgio}, Institute of Analysis and Scientific Computing, TU Wien, Wiedner Hauptstra\ss e 8-10, 1040 Vienna, Austria 
        
        \centerline{\sc and}

            \noindent MedUni Wien, W\"ahringer G\"urtel 18-20, 1090 Vienna, Austria.

        \noindent
	E-mail: \href{mailto:rossella.giorgio@student.tuwien.ac.at}{\tt rossella.giorgio@tuwien.ac.at}}
        
	\medskip \medskip \medskip \hrule  \medskip
   
    \small{
		\noindent
		{\sc L. Happ},
            Departamento de Matem\'aticas, 
            Universidad Aut\'onoma de Madrid, 
            C/ Francisco Tom\'as y Valiente 7,
            28049 Madrid, Spain

            \centerline{\sc and}

            \noindent{Institute of Analysis and Scientific Computing,
		TU Wien,
		Wiedner Hauptstra\ss e 8-10, 1040 Vienna, Austria.}
		
		\noindent
		E-mail:
            \href{mailto:leon.happ@uam.es}{\tt leon.happ@uam.es}
	}

\medskip \medskip \medskip \hrule  \medskip
        
	\small{		
		\noindent
		{\sc H. Sch\"onberger}, Research Institute in Mathematics and Physics, UCLouvain, Chemin du Cyclotron 2, 1348 Louvain-la-Neuve, Belgium

        \noindent
		E-mail: \href{mailto:hidde.schoenberger@tuwien.ac.at}{\tt hidde.schonberger@uclouvain.be}
	}
    \medskip \medskip \medskip \hrule  \medskip
}

\author{Rossella Giorgio}

\author{Leon Happ}
\author{Hidde Sch\"onberger}
\begin{document}
\allowdisplaybreaks

\thispagestyle{empty}

\thispagestyle{empty}
\begin{abstract}  
We study the homogenization of nonlocal micromagnetic functionals incorporating both symmetric and antisymmetric exchange contributions under the physical constraint that the magnetization field takes values in the unit sphere. Assuming that the nonlocal interaction range and the scale of heterogeneities vanish simultaneously, we capture the asymptotic behavior of the nonlocal energies by identifying their
$\Gamma$-limit, leading to an effective local functional expressed through a tangentially constrained nonlocal cell problem. %As a result, we also obtain the homogenization of the nonlocal Dzyaloshinskii–Moriya interaction (DMI), which has only recently been studied in the setting of pure localization. 
Our proof builds upon a tailored notion of two-scale convergence, which takes into account oscillations only in specific directions. It enables us to describe the two-scale limit of suitable nonlocal difference quotients, yielding a nonlocal analog of the classical limit decomposition result for gradient fields. To deal with the manifold constraint of the magnetization, we additionally prove that the microscopic oscillations in the two-scale limit are constrained to lie in the tangent space of the sphere.
%A key ingredient of our analysis is a variant of two-scale convergence that captures oscillations only in specific directions and results to be well-suited for the nonlocal setting.

% In this paper we study the homogenization of nonlocal micromagnetic functionals incorporating both symmetric and antisymmetric exchange contributions, using two-scale convergence techniques and under the physical constraint that the magnetization field takes values in the unit sphere. Starting from a nonlocal formulation of the micromagnetic energy, we analyze the asymptotic behavior of the energies as both the nonlocal interaction scale and the periodic heterogeneities simultaneously vanish. In this context, we establish the 
% $\Gamma$-convergence of the nonlocal energies to an effective functional expressed through a tangentially constrained, nonlocal cell problem. %As a result, we also obtain the homogenization of the nonlocal Dzyaloshinskii–Moriya interaction (DMI), which has only recently been studied in the setting of pure localization. 
% A key ingredient of our analysis is a variant of two-scale convergence that captures oscillations only in specific directions and results to be well-suited for the nonlocal setting.

\vspace{8pt}

% 49J45 Methods involving semicontinuity and convergence; relaxation
% 35B27 Homogenization in context of PDEs; PDEs in media with periodic structure
% 78M40 Homogenization in optics and electromagnetic theory?
% 47G20 Integro-differential operators
% 49S05 Variational principles of physics?
% Others?
\vspace{8pt}

\vspace{8pt}

%\noindent\textsc{Date:} \today.
\end{abstract}
\maketitle

\section{Introduction}\label{sec:intro}
In this paper we study the combined homogenization and localization of functionals of the form
\[
\Ecal_\epsilon(m):= \Fcal_\epsilon(m)+\Hcal_\epsilon(m),
\]
defined for every $\veps > 0$, with
\begin{align}
\begin{split}\label{eq:energies}
      \Fcal_\epsilon(m)&:=\int_\Omega\int_\Omega \alpha\left(\frac{x}{\epsilon},\frac{y}{\epsilon}\right) \frac{|m(y)-m(x)|^2}{|y-x|^2}\dd x\dd y \\
    \Hcal_\epsilon(m)&:=\int_\Omega\int_\Omega \beta\left(\frac{x}{\epsilon},\frac{y}{\epsilon}\right)\cdot\frac{(m(y)\times m(x))}{|y-x|}\dd x\dd y,
\end{split}
\end{align}
where $\Omega \subset \R^3$ is a bounded Lipschitz domain and $m \in L^2(\Omega;\S^{2})$. Specifically, we consider
\begin{align}
\begin{split}\label{eq:coeff}
\alpha\left(\frac{x}{\epsilon},\frac{y}{\epsilon}\right) &= a\left(\frac{x}{\epsilon},\frac{y}{\epsilon}\right) \frac{1}{\epsilon^3}\rho\left(\frac{y-x}{\epsilon}\right)\\
\beta\left(\frac{x}{\epsilon},\frac{y}{\epsilon}\right) &= \kappa\left(\frac{x}{\epsilon},\frac{y}{\epsilon}\right)  \frac{1}{\epsilon^3}\nu\left(\frac{y-x}{\epsilon}\right)
\end{split}
\end{align}
where $a$ and $\kappa$ are $Q$-periodic in both arguments with $Q:=(0,1)^3$ and describe the inhomogeneous microstructure inside $\Om$, while $\rho$ and $\nu$ are normalized kernels related to the localization effect (see Section~\ref{sec:setting} for more details). Our goal is to understand the effective macroscopic behavior of the energies $\Ecal_\epsilon$ as 
$\veps \to 0$, by analyzing the asymptotic properties of the associated nonlocal energy functionals \eqref{eq:energies}.
%; in light of this, one could also consider different scales $\epsilon$ and $\delta_\epsilon$ for both effects.
For the limit passage in the variational problem, we develop a theory for two-scale convergence of nonlocal difference quotients, which replace the role of gradients in this nonlocal framework.

The homogenization of the nonlocal \textit{symmetric} functionals $(\Fcal_\epsilon)_\epsilon$ has already been extensively studied, see~e.g.,~\cite{AAB23} and the references therein, albeit without the additional constraint that $m \in \S^2$ almost everywhere, and without resorting to the techniques of two-scale convergence; in contrast, the nonlocal \textit{antisymmetric} functionals $\Hcal_\epsilon$ have only recently been considered in the context of localization \cite{DDG24}. Our setting and both additional features are motivated by the theory of micromagnetics which constitutes the backdrop of the analysis in this paper. 

In our main Theorem~\ref{th:Gammalimit}, we prove the $\Gamma$-convergence of the functionals $\Ecal_\epsilon$ to the local functional
\[
\Ecal(m)\coloneqq\int_{\Omega} f_{\rm hom}(m(x),\nabla m(x)) \dd x \quad \text{for $m \in H^1(\Omega;\S^2)$},
\]
with the tangentially homogenized density
\begin{align}
\begin{split}\label{eq:tan_hom_cell}
 f_{\rm hom}(s,A) 
\coloneqq\inf_{v \in L^2_{\#,0}(Q;T_s \S^2)}\int_{\R^3}\int_{Q}&a(z,z+\xi)\rho(\xi)\frac{|A\xi +v(z+\xi)-v(z)|^2}{|\xi|^2}\\
&\quad+\kappa(z,z+\xi)\frac{(A\xi+v(z+\xi)-v(z))}{|\xi|} \cdot(s \times \nu(\xi))\dd z \dd \xi.
\end{split}
\end{align}
Up to the authors knowledge, this constitutes the first nonlocal homogenization result in the literature incorporating a manifold constraint.
%The main difficulties to overcome are exactly the size constraint of the magnetic field $|m|=1$ a.e.\ in $\Om$ and the analysis of the new antisymmetric term $\Hcal_\epsilon$.    \medskip

% Observations:
% \begin{itemize}
%     \item In \cite{AAB23}, the homogenization is first proved for very general functionals, with a cell formula for the limit density $f_{\rm hom}$. In the specific setting of quadratic forms, they use an abstract statement to deduce that the limit must be a quadratic form as well. However, this absract result does not immediately apply with the size constraint present. Hence, we maybe have to go through a different cell formula with the size constraint incorporated, and prove that it is quadratic directly, i.e., $f_{\rm hom}(\nabla m) = A \nabla m : \nabla m$ with $A$ a 2-tensor. \smallskip

%     \item For the antisymmetric term, does it behave somehow as lower order then $\Fcal_\epsilon$, so that maybe
%     \[
%     \Hcal_\epsilon (m_\epsilon) \to \sum_{i=1}^{3}\int_{\Omega} \nu_0\, m(x) \cdot (d_i\times \partial_i m(x))\dd x,
%     \]
%     and $\nu_0 >0$ the average of $\nu^{(1)}$?
% \end{itemize}

%\ROSS{I would move the "Outline of the paper" and the "Notation" after the Introduction/Motivation}

        %%%%%%%%%%%%%%%%%%%%%%%%%%%%%%%%%%%%%%%%%%%%%%%%%%%%%%%%%%%%%%%%%%%%%%%%%%%%%%%%%%%%%%%%%%%%%%%%%%%%%%

 \subsection*{Outline }\label{sec:outline}
    The paper is structured as follows.  
    In the rest of the Section \ref{sec:intro} we provide some background on the available literature about homogenization of convolution-type functionals, as well as on the theory of micromagnetics and related homogenization results. We also point out the novel features of the present homogenization result. Finally, we fix some notation.
    In Section~\ref{sec:setting} we formally define the nonlocal energies and we state the precise assumptions on the heterogeneous coefficients %\(a\) and \(\kappa\), 
    and on the localizing kernels. %\(\rho\) and \(\nu\). 
    In Section~\ref{sec:prelim} we recall the notion of two-scale convergence and its related properties. We also introduce a modified notion%of two-scale convergence
    , referred to as \(x\)-two-scale convergence, and show that its central features are inherited from the classical framework. Section~\ref{sec:compactness} focuses on the nonlocal $x$-two-scale compactness theorem and preliminary results.
    %, among which is the analogue of \cite[Proposition~1.14~(i)]{All92} in the \(x\)-two-scale sense. 
    Finally, in Section~\ref{sec:gammalim} we prove the main theorem about the $\Gamma$-convergence of nonlocal energies \eqref{eq:energies}. Moreover, in Appendix \ref{app:a} auxiliary computations on the nonlocal energies are presented.

%%%%%%%%%%%%%%%%%%%%%%%%%%%%%%%%%%%%%%%%%%%%%%%%%%%%%%%%%%%%%%%%%%%%%%%%%

\subsection{State of the art and contributions of the current work}\label{sec:background}
In this paper, we investigate a multiscale homogenization problem for nonlocal energy functionals comprising a convolution-type energy term and an antisymmetric contribution, motivated by the continuum theory of micromagnetics. We are interested in composite ferromagnetic materials occupying a domain $\Om$, with rapidly oscillating heterogeneities encoded in the coefficients $a$ and $\kappa$ in \eqref{eq:coeff}. These microscopic variations are coupled with localization effects occurring within the nonlocal micromagnetic exchange interactions \eqref{eq:energies}, all modulated by a small parameter $\veps > 0.$

\subsection*{Physics context and nonlocal formulation of micromagnetics}
% \subsection*{Nonlocal formulation of micromagnetics}

The continuum theory of micromagnetics was first established by Landau and Lifshitz \cite{LaLi35}, and later reformulated and extended by Brown \cite{Br63}. At the mesoscopic scale, and under the assumption that the temperature of the material  is constant and sufficiently far from a critical value known as the Curie point, magnetic phenomena within a ferromagnetic \textit{single-crystal} body occupying a region 
$\Om \subset \R^3$ are described by a vector field 
$M : \Om \to \R^3$ of constant magnitude $M_s$, referred to as the \textit{spontaneous magnetization}, which varies only in direction. Building on further contributions by Dzyaloshinskii \cite{Dz58} and Moriya \cite{Mo60}, the observable magnetic configurations in $\Om$ are modeled as the local minimizers of the micromagnetic energy functional
\begin{equation}
  \mathcal{G} (m) %\assign \mathcal{F} (m) +\mathcal{W} (m) \assign
  := 
  \int_{\Omega} a_{ex} | \nabla m|^2 + 
  \int_{\Omega} \kappa_{ex} \, \mathrm{curl}\,m \cdot m + \frac{1}{2}  \int_{\R^3} | h_{\mathsf{d}} [m] |^2 + \int_{\Omega} \varphi(m) - \int_{\Om} h_{\mathsf{a}} \cdot m.\label{eq:G}
\end{equation}
In this framework, the temperature is assumed to be constant, and the magnetization is normalized by $M_s$, yielding a magnetization variable $m \in H^1(\Om; \S^2).$ The first two terms in \eqref{eq:G} represent, respectively, the Dirichlet exchange energy, which favors spatial alignment of the magnetization, and the bulk Dzyaloshinskii-Moriya interaction (DMI), which introduces antisymmetric exchange effects that become relevant in ferromagnets with weak chiral symmetry. Here, $a_{ex} > 0$ and $\kappa_{ex} \in \R$ are material-dependent constants; the former penalizes spatial variations in magnetization, and the latter affects the chirality of the system \cite{BoYa89}.

The DMI energy was later incorporated into the micromagnetic energy functional to explain the formation of special magnetic configurations known as magnetic Skyrmions. These are localized chiral spin textures characterized by a non-trivial topological degree (cf.\ \cite{Me14, BoPa20}). Over the past years, they have attracted significant attention, mainly due to their potential applications in spintronics and other emerging magnetic technologies. In particular, one promising candidate for future data-storage devices is based on the exploitation of interactions between electron spins and Skyrmions (cf.\ \cite{FRC17} and the references therein). 
%This branch of research is concerned with the representation of data by means of the electron spin. In particular, one promising candidate for future data-storage devises is based on the exploitation of interactions between electron spins and Skyrmions (cf. \cite{FRC17} and the references therein). 

The last three terms in \eqref{eq:G} correspond, respectively, to the following contributions \cite{Br63, HuSch08}:
\begin{itemize}
    \item the magnetostatic self-energy, which quantifies the long-range dipolar interactions
among magnetic moments. The associated demagnetizing
field $h_\mathsf{d} [m]$ is characterized as the unique solution in $L^2 
(\mathbb{R}^3 ; \mathbb{R}^3)$ of the Maxwell--Ampére equations of
magnetostatics
\begin{equation*}
  \label{eq:Maxwell} \left\{\begin{array}{ll}
    \mathrm{div} (\mathds{1}_\Omega m +h_{\mathsf{d}} [m]) = 0 & \textrm{in }
    \mathbb{R}^3,\\
    \mathrm{curl}\, h_{\mathsf{d}} [m] = 0 & \textrm{in } \mathbb{R}^3 ;
  \end{array}\right.
  \end{equation*}
      \item the anisotropy energy, with density $\varphi: \S^2 \to [0, +\infty)$, assumed to be Lipschitz continuous and vanishing along material-preferred directions, known as \textit{easy axes};
      \item the Zeeman energy, where \(h_{\mathsf{a}} \in L^2(\Omega; \mathbb{R}^3)\) represents an externally applied magnetic field, modeling the tendency of the magnetization of the specimen to align with it.
  \end{itemize}
%The constant $\mu_0$ is the vacuum permeability.
However, in our analysis, we neglect these three contributions, as they act as 
$\Gamma$-continuous perturbations in the context of homogenization, see further justification below. Therefore, we restrict our attention solely to the exchange terms.

Regarding the exchange mechanisms, the classical theory of micromagnetics admits a natural \textit{nonlocal} interpretation, reflecting the inherently nonlocal nature of the discrete setting (cf.\ \cite{He28, Kef62})
\begin{equation} \label{eq:Heinsenberg}
  \mathcal{E}_{\Lambda} 
  :=  - S^2  \sum_{(i, j) \in \Lambda \times \Lambda} J_{ij} m_i \cdot m_j
  + \sum_{(i, j) \in \Lambda \times \Lambda} d_{ij} \cdot (m_i \times m_j),
\end{equation} where $J_{ij} = J_{ji}$ is the symmetric exchange constant between the spins
$m_i, m_j \in \S^2$, which occupy the $i$-th and $j$-th sites of the crystal
lattice $\Lambda$, and $S$ is the spin magnitude. Here, $d_{ij} = - d_{ji}$ is the axial Dzyaloshinskii vector which depends on  the relative position of the spins $m_i$ and
$m_j$ as well as on the symmetry of the crystal lattice. This is motivated by the isotropic Heisenberg model in the case when antisymmetric interactions cannot be neglected.

The increasing effort to reproduce nonlocal models at the continuum scale -- with the aim of obtaining more realistic and accurate descriptions of physical phenomena -- has progressively driven the scientific community toward the use of nonlocal-to-local approximation techniques. One of the first works in this direction (arising from purely theoretical considerations rather than specific applications) was established by Bourgain, Brezis, and Mironescu \cite{BBM01} through a pointwise approximation argument, in which the  Dirichlet energy in \eqref{eq:G} also receives a formal nonlocal justification. An important feature is the use of localizing kernels depending on a small parameter, which in the limit allows the recovery of the original local form of the energy. Subsequently, several extensions and improvements have been developed; for instance, in the \textit{symmetric} setting relevant to our case, significant contributions were made by Ponce \cite{Pon04, Pon04b}, in which the assumption of radiality for the kernels is removed, and a first $\Gamma$-convergence result is obtained. Moreover, motivated by micromagnetic considerations, the nonlocal counterpart of the \textit{antisymmetric} exchange term in \eqref{eq:G} has been recently established in \cite{DDG24}. This has led to the introduction of the following family of functionals as the micromagnetic exchange model
\begin{equation}\label{eq:localization}
 %\mathcal{F}_{\ep} (m) +\mathcal{H}_{\ep} (m) := 
 \int_{\Omega}\int_{\Omega} \rho_{\delta} 
     (x - y)  \frac{|m (x) - m (y) |^2}{|x - y|^2}\dd x \dd
     y + \int_{\Omega}\int_{\Omega} \nu_{\delta} 
     (x - y) \cdot \frac{(m (x) \times m (y))}{|x - y|} \dd x
     \dd y, 
\end{equation}
defined on a suitable subspace of $L^2 (\Omega ; \S^2)$, where $\left( \rho_{\delta}
\right)_{\delta}$ and $(\nu_{\delta})_{\delta }$ are families of suitable localizing kernels depending on the scaling parameter $\delta > 0.$ We observe here that the structure of the atomic energy model \eqref{eq:Heinsenberg} is preserved: Using the identity $|m_i - m_j|^2 = 2 - 2m_i \cdot m_j,$ we find the same variational model between the discrete and continuum setting, where the roles of $J_{ij}$ and $d_{ij}$ are substituted by the interaction kernels $\left( \rho_{\delta}
\right)_{\delta}$ and $(\nu_{\delta})_{\delta }$. Moreover, the Dzyaloshinskii vectors are directly connected to the kernel $\nu_\delta$ through the limit identity (cf.\ \cite{DDG24} and \eqref{eq:dmi_only_gen} below)
    \begin{equation*}
    d_i
    := \lim_{\delta \to 0} \int_{\R^3} \nu_{\delta} \left(\xi\right)\frac{\xi_i}{\abs{\xi}}\dd \xi \quad \text{for $i\in \{1,2,3\}$.}
    \end{equation*} 

In the case of a ferromagnetic body composed of \textit{several materials} -- as opposed to a single-crystal body -- the material-dependent parameters $a_{ex}$ and $\kappa_{ex}$ in \eqref{eq:G} (assuming once again that $M_s$ is normalized and also continuous across interfaces) 
are no longer constant within the region $\Om$ occupied by the ferromagnet. As a consequence, also the expression of the nonlocal exchange functional \eqref{eq:localization} must be modified to account for the effects of material inhomogeneities. This leads to the introduction of spatially varying coefficients of the form
\begin{equation*}
    a\left(\frac{x}{\veps},\frac{y}{\veps}\right) \quad \text{and} \quad \kappa \left(\frac{x}{\veps},\frac{y}{\veps}\right) \quad \text{for } x, y \in \R^3,
\end{equation*} which replace the constants $a_{ex}$ and $\kappa_{ex}$ in \eqref{eq:G}, with the additional parameter 
$\veps > 0$ encoding the scale of material heterogeneity. This provides a compelling motivation for the use of homogenization techniques in the setting of \textit{periodic} composite materials. 

\subsection*{Homogenization theory and convolution-type energies}

 The mathematical theory of homogenization in the calculus of variations has by now a long-standing history, see the monographs \cite{BrDe98,CioDo99, BeRy18} for a general overview. A historically significant contribution to the development of the current theory is the work by Marcellini \cite{Ma78}, where a general homogenization result for the class of convex densities was established. Here, the \(\Gamma\)-limit was proved with a functional density obtained via the so-called cell formula, derived from an auxiliary minimization problem on the unit cell. %Historically, a very important contribution to the theory was the work \cite{Ma78} by Marcellini, establishing a general homogenization result for the class of convex densities. 
% 	It was proved there that the \(\Gamma\)-limit provides a functional with a density that is obtained via the so-called cell formula, an auxiliary minimization problem on the unit cell.
This was followed by the independent works of Braides \cite{Br85} and M\"uller \cite{Mu87}, where the homogenization of general non-convex energy densities was rigorously developed.
	%resulting in a homogenize density in the form of a general asymptotic formula. 
Since then, the theory of homogenization has seen numerous substantial advancements. In connection with the goals of the present work, notably relevant is the recent progress in the homogenization of \textit{convolution-type functionals}, which, in particular, incorporate the case of nonlocal exchange interactions as substitutes for classical gradient terms. A comprehensive treatment of the $\Gamma$-limit behavior of convolution-type functionals is provided in \cite{AAB23} (see also \cite{BBD24} for the fractional case, \cite{BrP21} for the stochastic setting, and\cite{BrP22} for the scenario of periodically perforated domains). For a general picture of \textit{quadratic} convolution energies, which also motivates our analysis, we look at the nonlocal functional
\begin{equation}\label{eq:conv_hom_proto}	{\mathcal{F}}_{\veps,\delta}(u)\coloneqq 
		\int_\Om\int_\Om a\left(\frac{x}{\veps},\frac{y}{\veps}\right) \frac{1}{\delta^3}\rho\left(\frac{y-x}{\delta}\right)\frac{|u(y)-u(x)|^2}{|y-x|^2} \dl x\dl y,
	\end{equation}
    where $\veps, \delta > 0$ are multiscale parameters representing, respectively, the scale of material inhomo-\\geneities -- captured by the oscillatory coefficient $a$ in both of its entries -- and the interaction range for the kernel $\rho$. Throughout this work, we suppose that the scales \(\veps,\delta>0\) are comparable, i.e, that $\varepsilon$ and $\delta=\delta_\varepsilon$ satisfy \begin{equation}\label{eq:comp_scales}
	\delta_\varepsilon/\varepsilon \to \lambda \in(0,+\infty)
		\qquad \text{ as } \varepsilon \to 0,
	\end{equation}
    which means that the effects of localization and the oscillations caused by material heterogeneities %(due to impurities or material-dependent parameters in composites made of thin layers or with periodic inclusions)
    occur at the same scale.
    %In a more general setting than the quadratic case, 
    When $\delta_\epsilon = \epsilon$, it follows from the results in \cite{AAB23} that,
    %As a particular case of the results in \cite{AAB23}, it holds that, 
    under suitable assumptions on \(a\) and \(\rho\) (see also Remark~\ref{rmk:assumptions}), the family \(({\mathcal{F}}_{\veps,\delta})_\veps\) \(\Gamma\)-converges in the strong \(L^2\)-topology as \(\veps \to 0\) to the limit
	\begin{equation*}
		{\mathcal{F}}^{\rm conv}(u)
		\coloneqq \int_\Om f_{\hom}^{\rm conv}(\nabla u) \dl x,
	\end{equation*}
	with homogenized density
	\begin{equation}\label{eq:fconv-density}
		f_{\hom}^{\rm conv}(A)
		\coloneqq  \inf_{v \in L^2_{\#,0}(Q;\R^3)}\int_{\R^3}\int_{Q}a(z,z+\xi)\rho(\xi)\frac{|A\xi +v(z+\xi)-v(z)|^2}{|\xi|^2} \dl z\dl \xi.
	\end{equation}
 This generalizes the cell formula obtained in \cite{Ma78} to the nonlocal case by replacing the common gradient expression under the integral by its nonlocal counterpart. One might call the expression in \eqref{eq:fconv-density} a \textit{nonlocal cell formula}. 

	\subsection*{Homogenization in the framework of micromagnetics}

 The complex magnetic microstructure of ferromagnetic materials, in particular the emergence of magnetic domains, can be understood as the result of competing interactions at different length scales, arising from the various contributions in the micromagnetic energy functional. The homogenization of the fully classical micromagnetic energy functional was carried out in \cite{AlDF15}. Therein, the analysis, based on the two-scale convergence theory, focuses on the Dirichlet exchange energy
\begin{equation}\label{eq:only_ex}
		\int_\Om a\left(\frac{x}{\veps}\right)|\nabla m|^2 \dl x,
	\end{equation}
	for a coefficient \(a\in L^\infty(\R^3)\) that is $Q$-periodic and satisfies $\essinf_{\R^3}a>0$. %with $Q = (0,1)^3$. 
    It was shown that the remaining terms, namely the magnetostatic-self energy, the anisotropy energy, and the Zeeman energy, are \(\Gamma\)-continuous perturbations with respect to the strong $L^2$-topology and, for this reason, can be treated separately (cf.\  \cite[Proposition~6.20]{DM93}). In our nonlocal setting, we base our approach on this argument and focus exclusively on the exchange terms.
    
    Additionally, in \cite{AlDF15}, it was shown that the manifold constraint requiring the magnetization $m$ to be $\S^2$-valued almost everywhere in $\Om$ generally has a non-trivial effect on the homogenized limit energy density, as was thoroughly observed in \cite{BaMi10}. These findings imply that, for convex energy densities, the limit density is in general described by a tangentially homogenized cell formula, where the space of admissible test functions must be restricted to those functions only taking values in the tangent space \(T_m\S^2\). This feature reappears in \cite{AlDF15} and also in our result Theorem~\ref{th:Gammalimit} (see also \eqref{eq:tan_hom} below). Interestingly,  despite the \textit{non-convex} constraint on the manifold, the homogenization result reveals that the effective \(\Gamma\)-limit still admits a cell formula representation. 
    
    Later, in \cite{DaD20}, the authors studied the homogenization of an augmented version of the classical micromagnetic energy functional, which, in addition to the symmetric exchange energy in \eqref{eq:only_ex}, also incorporates the antisymmetric Dzyaloshinskii-Moriya interaction (DMI), which in a more general form reads as
 % \begin{equation}\label{eq:dmi_only}
	% \int_\Om  \kappa\left(\frac{x}{\veps}\right)m(x)\cdot \curl m(x) \dl x,
	% \end{equation}
	%where \(\kappa \in L^\infty(\R^3)\) is  $(0,1)^3$-periodic and represents the DMI coefficient. 
    %While the authors concentrated on the specific form of DMI known as bulk DMI, this contribution can take more general forms, expressed as 
    \begin{equation}\label{eq:dmi_only_gen}
		\sum_{i=1}^3\int_\Om \kappa\left(\frac{x}{\veps}\right)  \p_i m(x)\cdot (m(x)\times d_i) \dl x,
	\end{equation} where \(\kappa \in L^\infty(\R^3)\) is  $Q$-periodic and represents the DMI coefficient, and for \( i\in\{1,2,3\}\)
	\(d_i\in \R^3\)  are the Dzyaloshinskii vectors, depending on the underlying crystal structure of the material; note that when $d_i$ equals the canonical basis vector $e_i$, we find again the bulk DMI as in \eqref{eq:G}. However, the results in \cite{DaD20} %(see also Section~4 therein)
    also apply to this general case and it follows that the functional
	\begin{equation*}
		\int_\Om a\left(\frac{x}{\veps}\right)|\nabla m|^2 \dl x
		+\sum_{i=1}^3\int_\Om \kappa\left(\frac{x}{\veps}\right)  \p_i m(x)\cdot (m(x)\times d_i) \dl x
	\end{equation*}
	\(\Gamma\)-converges with respect to the weak \(H^1\)-topology to 
	\begin{equation*}
		{\mathcal{E}}^{\rm tan}(m)
		\coloneqq \int_\Om f_{\hom}^{\rm tan}(m,\nabla m) \dl x
	\end{equation*} 
	with density
	\begin{equation}\label{eq:tan_hom}
		f^{\rm tan}_{\hom}(s,A)
		\coloneqq \inf_{v \in H^1_{\#,0}(Q;T_s \mathbb{S}^2)}\int_{Q}a(z)|A+\nabla v(z)|^2 
		+\sum_{i=1}^3\kappa(z)(A+\p_i v(z))\cdot (s\times d_i) \dl z.
	\end{equation}

 % see also \cite{AlDF15}), which we modify by considering nonlocal interactions.
 %Rooted in the physically reasonable assumption that under a certain temperature threshold, known as the Curie temperature, the magnetization in a ferromagnet is homogeneously saturated, it is an important ingredient in the theory that the relevant magnetization vector fields are assumed to have constant norm. Therefore, we adopt the standard convention to treat the magnetizations as elements of the function space \(H^1(\Om;\mathbb{S}^2)\), where the target manifold constraint is understood pointwise almost everywhere.
%We stress that the manifold constraint imposed on the space \(H^1(\Om;\mathbb{S}^2)\) has in general a non-trivial effect on the homogenized limit energy density, as was thoroughly studied in \cite{BaMi10}. 
     %Remarkably, due to the manifold constraint, the space \(H^1(\Om;\mathbb{S}^2)\) is not convex anymore; however, the homogenization reveals that the effective \(\Gamma\)-limit still admits a cell formula representation. 

	\subsection*{Contributions of the current paper} Following the homogenization frameworks developed in \cite{AlDF15, DaD20} for the local micromagnetic setting, the novelty of our work lies in extending the analysis to the more general nonlocal micromagnetic model. For the functionals $\mathcal{F}_{\veps,\delta}(m) $ in \eqref{eq:conv_hom_proto} and the new antisymmetric exchange functionals
    \begin{equation*}\label{eq:conv_hom_dmi}
		 {\mathcal{H}}_{\veps, \delta}(m)
		 :=\int_\Omega\int_\Omega \kappa\left(\frac{x}{\veps},\frac{y}{\veps}\right)  \frac{1}{\delta^3}\nu\left(\frac{y-x}{\delta}\right)\cdot\frac{m(y)\times m(x)}{|y-x|}\dl x\dl y,
	\end{equation*}
    %$\mathcal{H}_{\veps,\delta}(m) $ in \eqref{eq:conv_hom_dmi},
    with parameters satisfying \(\delta=\veps\), we develop a full homogenization theory over the space \(L^2(\Om;\mathbb{S}^2)\). This leads to the problem introduced in \eqref{eq:energies} above, where the joint localization and homogenization as \(\veps\to 0\) is the subject of our main Theorem~\ref{th:Gammalimit}. In particular, we derive that the corresponding \(\Gamma\)-limit can be represented via the tangentially homogenized \textit{nonlocal cell formula} \eqref{eq:tan_hom_cell}. As a by-product of our analysis, we also obtain an extension of the result to the case of \textit{comparable scales}, as defined in \eqref{eq:comp_scales} (see also Remark~\ref{rmk:diff_scales} for further details). %Up to the authors knowledge, this is the first nonlocal homogenization result in the literature incorporating a manifold constraint.
    
 In our proof strategy, we follow a different approach from that of \cite{AAB23},  where the convex homogenization formula was deduced as a special case of a more general framework based on an abstract localization scheme for establishing \(\Gamma\)-convergence (see also \cite[Chapters~14--20]{DM93} and \cite[Sections~10--12]{BrDe98}). Instead, we employ the theory of two-scale convergence (dating back to the seminal work \cite{Ng89} and further developed in \cite{All92}),
 which is particularly well-suited for the analysis of convex homogenization problems. Although we are dealing with non-convex manifold constraints, two-scale convergence remains a particularly effective tool for our endeavor. At the same time, its application to the nonlocal setting considered in this paper is not straightforward and requires the development of specific notions and auxiliary results tailored to this context.
 In particular, the nonlocal framework necessitates a \textit{variant} of two-scale convergence, denoted as \(x\)-two-scale convergence, which takes into account oscillations in only certain components of the variables. This idea is not entirely new, see~e.g.~\cite[Section~6]{Ne10}, %and although it is not difficult to extend the classical two-scale convergence results to this framework,
 but we still dedicate Section~\ref{sec:prelim} to carefully introduce the concept and to derive the auxiliary results necessary for our application.

 For the identification of the two-scale limit of suitable nonlocal difference quotients, which replace the classical gradient in our nonlocal framework, we show in Theorem~\ref{th:2scale} a nonlocal analog of the limit decomposition originally proved in \cite[Proposition 1.14(i)]{All92}. The classical case concerns two-scale convergent sequences that are additionally uniformly bounded in \(H^1(\Om;\mathbb{S}^2)\), and exploits the fact that the orthogonal complement of divergence-free functions are exactly gradient fields. In the nonlocal case, we instead draw on the theory of unbounded linear operators to derive a corresponding decomposition, resulting, in a heuristic sense, for a nonlocal counterpart of the classical divergence operator (see also the auxiliary results Lemma~\ref{le:Aoperator} and Lemma~\ref{le:densitykernel}).
 %In our equi-coercivity result in Theorem~\ref{th:2scale}, we also show a nonlocal analogue of the limit decomposition from \cite[Proposition1.14(i)]{All92}, concerning two-scale convergent sequences that are additionally uniformly bounded in \(H^1(\Om;\mathbb{S}^2)\).
 We believe this result is of independent interest and may be valuable for future investigations in the context of nonlocal homogenization. Relying on this derivation, we then suitably adapt the approach in \cite{DaD20} to the nonlocal setting in order to obtain a tangentially homogenized nonlocal cell formula for the \(\Gamma\)-limit density. As an intermediate step, we provide in Proposition~\ref{prop:min_fhom} an equivalent characterization of the unique solution to the limit cell problem.

	%   Moreover, in Theorem~\ref{th:2scale} we prove a nonlocal analogue of the limit decomposition in \cite[Proposition~1.14~(i)]{All92} of two-scale converging sequences that are additionally uniformly bounded in \(H^1(\Om;\mathbb{S}^2)\), which, as we believe, is of general interest for further studies of nonlocal homogenization problems. With these results at hand, we suitably adapt the approach in \cite{DaD20} to the nonlocal framework to derive a tangentially homogenized nonlocal cell formula for the \(\Gamma\)-limit density. As an intermediate step, we provide in Proposition~\ref{prop:min_fhom} an equivalent characterization of the unique solution for the limit cell problem.

    \subsection*{Notation}\label{Sec:Notation}

We fix the notation for the rest of the paper. We write \(\cdot\) for the standard Euclidean inner product and \(\colon\) for the Frobenius inner product between matrices. 
Throughout the paper, 
$\Omega \subset \R^3$ is a fixed bounded Lipschitz domain, and  for \(\mu>0\) we define \(\Om_{\mu}\coloneqq \{x \in \Omega : \operatorname{dist}(x, \partial \Om) > \mu\}\). The unit cube in three dimensions is denoted by \(Q\coloneqq (0,1)^3\subset \R^3\).  For a set $E \subset \R^n$, the characteristic function is denoted by
\[
\mathds{1}_{E}(x)=\begin{cases}
    1 &\text{if $x \in E$},\\
    0 &\text{if $x \in \R^n \setminus E$.}
\end{cases}
\]
%The non-negative real numbers are denoted by $\R^{+}:=[0,\infty)$. 

The subscript \(\#\) will always denote that a function is periodic in the according unit cube; for example, the space \(L^2(\Om;L^2_\#(Q;\R^3))\) is understood as the periodic extension of functions in \(L^2(\Om;L^2(Q;\R^3))\) to \(L^2(\Om;L^2(\R^3;\R^3))\). Similarly, by \(L^\infty_{\#}(Q\times Q)\) we refer to functions which are periodic in both of their entries, or equivalently, in \(Q\times Q\subset \R^6\). Moreover, the subscript \(0\) stands for a vanishing mean value. Accordingly, the space \(L^2(\Om;L^2_{\#,0}(Q;\R^3))\) will contain those functions \(u\in L^2(\Om;L^2_{\#,0}(Q;\R^3))\) for which additionally
\begin{equation*}
    \inner{u}_{Q}(x)
    \coloneqq \frac{1}{Q}\int_Q u(x,z) \dl z = 0
\end{equation*}
for a.e.~\(x\in \Om\). 

 For the two-dimensional unit sphere \(\S^2\subset \R^3\), % we define the Sobolev space
% %\begin{equation*}
%     H^1(\Omega;\S^2)
%     \coloneqq \{m\in H^1(\Omega;\R^3):\, m(x)\in \S^2
%     \text{ for a.e. } x\in \Omega\}.
% \end{equation*}
we set \(T_s\S^2\) as the tangent space to \(\S^2\) at a given point \(s\in \S^2\). For a function \(m\in H^1(\Omega;\S^2)\), the space \(L^2(\Om;L^2_{\#,0}(Q;T_m\S^2))\) is understood as the space of functions \(w\in L^2(\Om;L^2_\#(Q;\R^3))\) such that \(w(x)\in T_{m(x)}\S^2\) for a.e.~\(x\in \Omega\). For the sake of notation, we suppress in \(L^2(\Om;L^2_{\#,0}(Q;T_m\S^2))\) the implicit dependence on \(x\in \Omega\). 

% \hidde{Separate notation for the extension of a function by zero or not?}
% \comment{I don't think that we ned it. We only use it in the proofs of Lemma~\ref{le:densitykernel} and Proposition~\ref{prop:density}, right?}

To denote weak and strong two-scale convergence (see Definition~\ref{def:ts}) we will use the symbols $\twoscale$ and $\wtwoscale$, respectively. In the case of \(x\)-two-scale convergence (see Definition~\ref{def:xts}), we indicate the variable dependence with an additional superscript and employ the notations $\xtwoscale$ and $\wxtwoscale$.

In the rest of the paper, \(C>0\) will denote a generic constant that may change from line to line.

%%%%%%%%%%%%%%%%%%%%%%%%%%%%%%%%%%%%%%%%%%%%%%%%%%%%%%%%%%%%%%%%%%%%%%%%%%%%%%%%%%%%%%%%%%%%%%%%%%%%%%

%%%%%%%%%%%%%%%%%%%%%%%%%%%%%%%%%%%%%%%%%%%%%%%%%%%%%%%%%%%%%%%%%%%%%%%%%%%%%%%%%%%%%%%%%%%%%%%%%%%%%%

\section{Setting and Assumptions}\label{sec:setting}
For $m \in L^2(\Om; \R^3)$ and $\veps > 0$, we consider the nonlocal energy 
\begin{equation}\label{eq:full-energy}
    \Ecal_\epsilon(m) = \int_\Omega\int_\Omega \alpha\left(\frac{x}{\epsilon},\frac{y}{\epsilon}\right) \frac{|m(y)-m(x)|^2}{|y-x|^2} + \beta\left(\frac{x}{\epsilon},\frac{y}{\epsilon}\right)\cdot\frac{(m(y)\times m(x))}{|y-x|} \dd x\dd y,
\end{equation}
with coefficients
\begin{equation}\label{coefficients}
    \alpha\left(\frac{x}{\epsilon},\frac{y}{\epsilon}\right) = a\left(\frac{x}{\epsilon},\frac{y}{\epsilon}\right) \frac{1}{\epsilon^3}\rho\left(\frac{y-x}{\epsilon}\right) \quad \text{and } \quad 
\beta\left(\frac{x}{\epsilon},\frac{y}{\epsilon}\right) = \kappa\left(\frac{x}{\epsilon},\frac{y}{\epsilon}\right)  \frac{1}{\epsilon^3}\nu\left(\frac{y-x}{\epsilon}\right).
\end{equation}
The first term of $\Ecal_\epsilon$ in \eqref{eq:full-energy} is denoted by 
$\Fcal_\epsilon$ and referred to as the \textit{symmetric energy}, while the second term, denoted by 
$\Hcal_\epsilon$, is called the \textit{antisymmetric energy}. In most parts of the analysis, $m$ also satisfies the micromagnetic constraint, that is $|m| = 1$ a.e.\ in $\Om$.

Motivated by technical considerations, we impose the following assumptions on the oscillating coefficients and the localizing kernels.

\begin{enumerate}[label=(H\arabic*)]
    \item\label{H1} \textbf{(Oscillating coefficients)}
        The coefficients satisfy $a, \kappa \in L^{\infty}_{\#}(Q\times Q)$ with $a(x,y) \geq a_0 >0$ for a.e.~$x,y \in \R^3$.
    \item\label{H2} \textbf{(Kernel of the symmetric energy)}
        The kernel $\rho \in L^1(\R^3)$ is non-negative, satisfies $\|\rho\|_{L^1(\R^3)}=1$ and
        \begin{equation}\label{eq:rhocoercive}
            \essinf_{\xi \in B_r(0)}\frac{\rho(\xi)}{|\xi|^2}>0 \quad \text{for some $r>0$.} 
        \end{equation} 
    \item\label{H3} \textbf{(Kernel of the antisymmetric energy)}
        The kernel $\nu \in L^1(\R^3;\R^3)$ is chosen such that $\|\nu\|_{L^1(\R^3; \R^3)}=1$.
    \item\label{H4} \textbf{(Interaction of kernels)}
        We assume the control
    \begin{equation}\label{eq:ratiobound}
    \norm*{\frac{\nu}{\rho^{1/2}}}_{L^{2}(\R^3)}<+ \infty.
        % \norm*{\frac{\nu}{\rho}}_{L^{\infty}(\R^3)}<+ \infty.
    \end{equation}
\end{enumerate}
% In this case, the Dzyaloshinskii vectors are simply given by
% \[
% d_i:=\int_{\R^3} \nu(y)\frac{y_i}{|y|}\dd y \in \R^3 \quad \text{for $i=1,2,3$.}
% \]
\begin{remark}[On Hypotheses \ref{H2}-\ref{H4}]\label{rmk:assumptions}
We work with kernels of the form
\begin{equation}\label{kernels}
\rho_\veps(\xi) := \frac{1}{\veps^3}\rho \left(\frac{\xi}{\veps} \right) \quad \text{and} \quad \nu_\veps(\xi) := \frac{1}{\veps^3}\nu \left(\frac{\xi}{\veps} \right), 
\end{equation} 
for which the localization condition
\begin{equation*}
\lim_{\veps \to 0} \int_{|\xi| > \delta} \rho_\veps(\xi) \dd \xi = 0 \quad \text{and} \quad \lim_{\veps \to 0} \int_{|\xi| > \delta} |\nu_\veps(\xi) |\dd \xi = 0 \quad \text{for every } \delta > 0
\end{equation*} 
is automatically satisfied.
As a consequence, it is sufficient to pose assumptions directly on the functions $\rho$ and $\nu$, as specified in \ref{H2}-\ref{H4}.

Assumption \ref{H2} for the kernel $\rho$ is quite natural, reflecting the classical integrability condition as in \cite{BBM01}. The coercivity-type estimate in \eqref{eq:rhocoercive} can be viewed as a relaxation of radiality (cf. \cite{BBM01}), and still allows us to exploit compactness arguments (see Theorem \ref{th:2scale}) and a nonlocal Poincaré-type inequality (see \eqref{eq:nonlocal_poincare}).
Moreover, Assumption \ref{H2} shares similarities with the hypotheses introduced in \cite[Section 2.3]{AAB23} for the homogenization of convolution-type energies. However, we note that these conditions might not be optimal. Indeed, in the context of the pointwise localization, optimal assumptions have been identified in \cite{DDP24} for the case $p=2$, in \cite{Fog25} for all $p \in [1, +\infty)$ under the assumption of radiality, and in \cite{GeS25} for the more general setting.

Assumptions \ref{H3}-\ref{H4} for the kernel $\nu$ are inspired by the analysis in \cite{DDG24}. In particular, Assumption \ref{H4} reflects the fact that the kernels $\rho$ and $\nu$ interact with one another. Nevertheless, due to the specific structure of the kernels in \eqref{kernels}, we do not need to impose any additional conditions to ensure the finiteness of quantities of the form
\begin{equation*}
d_i
    =\int_{\R^3} \nu\left(\xi\right)\frac{\xi_i}{\abs{\xi}}\dd \xi, \quad \text{$i\in \{1,2,3\}$},
\end{equation*}
or, in particular, of the averaged DMI vectors, $\bar{d_i}$, introduced later in \eqref{eq:dmivectors} for Proposition \ref{prop:min_fhom}.

\end{remark}

\begin{remark}[On the symmetry of kernels and the well-definedness of the energies]
    In our analysis, the two kernels $\rho$ and $\nu$ are not assumed to satisfy any symmetry conditions.  However, in the context of nonlocal energies of the form  $\Fcal_\veps$ and $\Hcal_\veps$, and in comparison with the pure localization setting (see, e.g., \cite{BBM01, DDG24}), it is common to assume that $\rho$ is  even and that $\nu$ is odd. We note that, without loss of generality, the analysis can always be reduced to this case.
    
   Moreover, in the homogenization framework considered here, where the oscillating coefficients $a$ and $\kappa$ play a role, these symmetry assumptions naturally translate into requiring that the functions $\alpha$ and $\beta$ in \eqref{coefficients} are, respectively, symmetric and antisymmetric, that is,
   \begin{align*}
\alpha\left(\frac{x}{\epsilon},\frac{y}{\epsilon}\right) = \alpha\left(\frac{y}{\epsilon},\frac{x}{\epsilon}\right) \quad \text{and} \quad
\beta\left(\frac{x}{\epsilon},\frac{y}{\epsilon}\right) = - \beta\left(\frac{y}{\epsilon},\frac{x}{\epsilon}\right) \quad \text{for every } x, y \in \R^3.
\end{align*}
Therefore, in the case of $\beta$, for instance, one may choose the physical assumption that $\kappa$ is symmetric and the kernel $\nu$ is odd.
\end{remark}

In the following, it will be useful to rewrite the nonlocal energies as 
\begin{align}
    \label{Feps-shifted}\Fcal_\epsilon(m)&=\int_{\R^3 } \int_{\Omega_{\veps \xi}} a\left(\frac{x}{\epsilon},\frac{x}{\epsilon} + \xi \right) \rho(\xi)  \frac{|m(x + \veps \xi)-m(x)|^2}{|\veps \xi|^2}\dd x\dd \xi, \\
    \label{Heps-shifted}\Hcal_\epsilon(m)&=\int_{\R^3 } \int_{\Omega_{\veps \xi}}\kappa \left(\frac{x}{\epsilon},\frac{x}{\epsilon} + \xi \right) \nu(\xi) \cdot\frac{(m(x + \veps \xi)\times m(x))}{|\veps \xi|}\dd x\dd \xi,
\end{align}
where we performed the change of variable $y = x + \veps \xi$, with $\xi \in \R^3$ and where $ 
\Omega_{\veps \xi}= \{x \in \Omega : x + \veps\xi \in \Omega \}.  $ 
By expressing the functionals in this form, the kernels $\rho$ and $ \nu $ are no longer dependent on the localization parameter $\veps > 0, $ which instead now explicitly interacts with the magnetization $m$. Furthermore, oscillations are observed solely in the variable $x$, and not in $\xi$. This leads to the introduction in Section \ref{sec:prelim} of a suitable notion of two-scale convergence that takes into account oscillations in only specific variables of the functions.

% \begin{enumerate}[label = (G\arabic*)]
% \item \label{itm:g1} It holds that $\rho \geq 0$ on $\R^3$ and $\int_{\R^3} \rho\dd x =1$.
% \end{enumerate}

%%%%%%%%%%%%%%%%%%%%%%%%%%%%%%%%%%%%%%%%%%%%%%%%%%%%%%%%%%%%%%%%%%%%%%%%%%%%%%%%%%%%%%%%%%%

\section{Preliminaries on two-scale convergence}\label{sec:prelim}

%\subsection{Two-scale convergence}

We recall here the definition of two-scale convergence (cf. \cite[Definition~1]{LNW02}, and also \cite{Ng89,All92}). Moreover, in view of the fact that the functions \(a\) and  \(\kappa\) in \eqref{Feps-shifted} and \eqref{Heps-shifted} only capture oscillations in the \(x\) and not the \(\xi\) variable, we introduce a generalized notion of two-scale convergence, in which we only account for oscillations appearing in a certain set of components. A similar idea was used for example in \cite{Ne10} in the context of dimension reduction of thin films with lateral periodic inhomogeneities. We show  that the important features of classical two-scale convergence naturally remain valid in this modified setting.

For the rest of this paragraph, let \(m,n\in \N\) and let \(E\) be an open subset of \(\R^n\). We denote by \(\eta\) the variable living on \(E\), by \(Q_\eta\) the \(n\)-dimensional unit cube \((0,1)^n\subset \R^n\), and by \(\zeta\) the variable associated with \(Q_\eta\). Anticipating later parts of the paper and aiming to avoid notational confusion, we point out that in Section~\ref{sec:compactness} and thereafter, the theory developed here will be applied in the case \(n=6\), where the cube \(Q=(0,1)^3\) will play the role of \(Q_x\) (introduced below) and serves as a representative lower-dimensional slice of \(Q_\eta\).

\begin{definition}[Two-scale convergence]\label{def:ts}
     Let \(\sequence{v}{\veps}\subset L^2(E;\R^m)\) and \(v\in L^2(E\times Q,\R^m)\).
    \begin{enumerate}[label = (\roman*)]
        \item 
            The sequence \(\sequence{v}{\veps}\) is said to weakly two-scale converge in $L^2(E; \R^m)$ to the limit function \(v\), written as $v_\epsilon \wtwoscale v$, if
            \begin{equation}\label{eq:w_ts}
                \lim_{\veps\to 0}\int_{E} v_\veps(\eta) \cdot \psi(\eta,\eta/\veps)\dl \eta
                =\int_{E} \int_{ Q_\eta} v(\eta,\zeta)\cdot\psi(\eta,\zeta)\dl \zeta\dl \eta
            \end{equation}
            for all \(\psi\in L^2(E;C_{\#}(Q_\eta;\R^m))\). \smallskip
        \item\label{it:sts}
            If in addition to \eqref{eq:w_ts} there holds
            \begin{equation}\label{eq:s_ts}
                \lim_{\veps\to 0}\norm{v_\veps}_{L^2(E)}
                = \norm{v}_{L^2(E\times Q_\eta)},
            \end{equation}
            we say that \(\sequence{v}{\veps}\) converges strongly two-scale in $L^2(E; \R^m)$ to \(v\) and write $v_\epsilon \twoscale v$.
    \end{enumerate}
\end{definition}

Very importantly, all functions of the form \(v_\veps(\eta)\coloneqq \psi(\eta,\eta/\veps)\) with \(\psi \in L^2(E;C_{\#}(Q_\eta;\R^m))\) actually converge strongly two-scale to \(v(\eta,\zeta)\coloneqq \psi(\eta,\zeta)\) as $\veps \to 0$ (cf. \cite[Example~1 and Theorem~11]{LNW02} or also \cite{Vi06}, where this is deduced with the help of the periodic unfolding). We also note that the two-scale limit is unique (cf. \cite{LNW02} below Definition~1).

Moreover, two-scale convergence enjoys the following pivotal properties: 
\begin{enumerate}[label=(P\arabic*)]
    \item\label{P1} Every bounded sequence in \(L^2(E;\R^m)\) admits a weakly two-scale converging subsequence (cf.\ \cite[Theorem~7]{LNW02}). 
    
    \item\label{P2} If \(\sequence{v}{\veps}\subset L^2(E;\R^m)\) weakly two-scale converges to \(v\in L^2(E\times Q_\eta;\R^m)\), there holds (cf.\ \cite[Theorem~10]{LNW02})
\begin{equation}\label{eq:norm_lsc}
    \liminf_{\veps\to 0} \norm{v_\veps}_{L^2(E)}
    \ge \norm{v}_{L^2(E\times Q_\eta)}.
\end{equation}

\item\label{P3} If \(\sequence{u}{\veps},\sequence{v}{\veps}\subset L^2(E;\R^m)\) converge, respectively, weakly two-scale and strongly two-scale to limit functions \(u,v\in L^2(E\times Q_\eta,\R^m)\), then 
\begin{equation}\label{eq:ts_products}
    \lim_{\veps\to 0}\int_{E} u_\veps(\eta) \cdot v_\veps(\eta)\dl \eta
        =\int_{E} \int_{Q_\eta} u(\eta,\zeta)\cdot v(\eta,\zeta)\dl \zeta\dl \eta.
\end{equation}
\end{enumerate}
This last identity follows immediately by reformulating two-scale convergence as  \(L^2\)-convergence in the space \(L^2(\R^n\times Q_\eta;\R^m)\) via periodic unfolding (cf.\ again \cite{Vi06}; see also \cite{CDG02}), and by applying the corresponding properties on products of weakly and strongly converging sequences in \(L^2(\R^n\times Q_\eta;\R^m)\)  (see also \cite[Proposition~2.5 and Proposition~2.9]{Vi06}).

%We would like to apply the theory of two-scale convergence to functions that live in the open subset \(A=\Om\times \R^3\subset \R^6\) and range in \(\R^3\). In particular, in view of the oscillations in the function \(a\) in \eqref{eq:feps_repar}, we study now the behavior of functions of type \(v(x,\xi)\), where we take \(\eta=(x,\xi)\in \Om\times \R^3\), when paired with test functions that only oscillate in the variable \(x\). 
For the rest of the section, we adopt the following notation. We decompose the variable $\eta \in \R^n$ as $\eta=(x,\xi) \in \R^{n_{x}}\times\R^{n_\xi}$ with \(n=n_x+n_\xi\), and consider a set of the form \(E=E_x \times E_\xi\). Moreover, for \(\zeta \in Q_\eta\) we write \(\zeta=(z,\omega)\in Q_x\times Q_\xi\), where  \(Q_x\) and \(Q_\xi\) denote the unit cubes in $\R^{n_x}$ and $\R^{n_\xi}$, respectively. With this, we now augment Definition~\ref{def:ts} as follows.
\begin{definition}[\(x\)-two-scale convergence]\label{def:xts}
    Let \(\sequence{v}{\veps}\subset L^2(E;\R^m)\) and \(v\in L^2(E_x\times E_\xi\times Q_x;\R^m)\).
    \begin{enumerate}[label = (\roman*)]
        \item 
            The sequence \(\sequence{v}{\veps}\) is said to weakly \(x\)-two-scale converge in $L^2(E;\R^m)$ to the limit function \(v\), written as $v_{\epsilon} \wxtwoscale v$, if
            \begin{equation}\label{eq:w_xts}
                \lim_{\veps\to 0}\int_{E_x}\int_{E_\xi} v_\veps(x,\xi) \cdot \psi(x,\xi,x/\veps)\dl \xi\dl x
                =\int_{E_x}\int_{E_\xi}\int_{Q_x} v(x,\xi,z)\cdot\psi(x,\xi,z)\dl z\dl \xi\dl x
            \end{equation}
            for all \(\psi\in L^2(E;C_{\#}(Q_x;\R^m))\). \smallskip
        \item 
            If in addition to \eqref{eq:w_xts} there holds
            \begin{equation}\label{eq:s_xts}
                \lim_{\veps\to 0}\norm{v_\veps}_{L^2(E)}
                = \norm{v}_{L^2(E_x\times E_\xi\times Q_x)},
            \end{equation}
            we say that \(\sequence{v}{\veps}\) converges strongly \(x\)-two-scale to \(v\) and write $v_\epsilon \xtwoscale v$.
    \end{enumerate}
\end{definition}
As the two-scale limit, also the \(x\)-two-scale limit is unique. 
We define for \(\tilde{v}\in L^2(E\times Q_\eta;\R^m)\) 
\begin{equation*}
    \inner{\tilde{v}}_{Q_\xi}(x,\xi,z)
    \coloneqq \int_{Q_\xi}\tilde{v}(x,\xi,z,\omega)\dl \omega
\end{equation*}
as the average over \(Q_\xi\). Note that this is always well-defined, since Fubini's theorem and the boundedness of \(Q_\xi\) guarantee that \(\tilde{v}(x,\xi,z,\cdot)\in L^1(Q_\xi;\R^m)\) for a.e.~\((x,\xi,z)\in E_x\times E_\xi\times Q_x\). 

The following lemma shows a hierarchy between two-scale convergence and \(x\)-two-scale convergence, which follows readily from Definition~\ref{def:ts} and Definition~\ref{def:xts} (cf.\ also \cite[Proposition~6.2.5]{Ne10}).
\begin{lemma}\label{lem:ts_and_xts}
Let \(\sequence{v}{\veps}\subset L^2(E;\R^m)\).
Then:
\begin{enumerate}[label=(\roman*)]
        \item\label{it:ts_to_xts}
        If $(v_\epsilon)_{\epsilon}$ weakly two-scale converges to some $\tilde{v}\in L^2(E\times Q_\eta;\R^m)$, then $(v_\epsilon)_\epsilon$ weakly $x$-two-scale converges to $\inner{\tilde{v}}_{Q_\xi}$. \smallskip
        
        \noindent Conversely, if $(v_\epsilon)_\epsilon$ weakly $x$-two-scale converges to $v \in L^2(E_x\times E_\xi \times Q_x;\R^m)$, then (up to a subsequence) $(v_\epsilon)_{\epsilon}$ weakly two-scale converges to some $\tilde{v}\in L^2(E\times Q_\eta;\R^m)$ with $\inner{\tilde{v}}_{Q_\xi}=v$. \smallskip
        
        \item\label{it:xts_to_ts}
        The sequence $(v_\epsilon)_\epsilon$ strongly two-scale converges to some $\tilde{v} \in L^2(E\times Q_\eta;\R^m)$ constant in the $\omega$-variable if and only if $(v_\epsilon)_\epsilon$ strongly $x$-two-scale converges to $\tilde{v}$, seen as an element of $L^2(E_x\times E_\xi \times Q_x;\R^m)$.
    \end{enumerate}
    % \begin{enumerate}[label=(\roman*)]
    %     \item\label{it:ts_to_xts}
    %          Two-scale convergence implies \(x\)-two-scale convergence in the sense that
    %     \begin{align*}
    %     &\sequence{v}{\veps} \text{ converges weakly (resp.\ strongly) two-scale to some } \tilde{v}\in L^2(E\times Q;\R^m)\\
    %     &\Longrightarrow 
    %     \sequence{v}{\veps} \text{ converges weakly (resp.\ strongly) }x\text{-two-scale to } \inner{\tilde{v}}_{Q_\xi}.
    %     \end{align*}
    %     \item\label{it:xts_to_ts}
    %         \(x\)-two-scale convergence implies two-scale convergence in the sense that
    %     \begin{align*}
    %     &\sequence{v}{\veps} \text{ converges weakly (resp.\ strongly) } x\text{-two-scale to some } v \in L^2(E_\xi\times E_x\times Q_x,\R^m)\\
    %     &\Longrightarrow 
    %      \text{ there exists a subsequence of } \sequence{v}{\veps} \text{ that converges weakly (resp.\ strongly) }\text{two-scale }\\
    %     &\qquad\,\,\,\text{to some } \tilde{v}\in L^2(E\times Q;\R^m) \text{ with } \inner{\tilde{v}}_{Q_\xi}=v 
    %     \text{ a.e. in } E_\xi\times E_x\times Q_x.
    % \end{align*}
    % \end{enumerate}
\end{lemma}
\begin{proof}
    \textit{Part (i).}
    Let $v_\epsilon \wtwoscale \tilde{v}$ in $L^2(E;\R^m)$ as $\epsilon \to 0$. By choosing in \eqref{eq:w_ts} test functions $\psi\in L^2(E;C_{\#}(Q_x;\R^m))$, we have
    \begin{equation*}
        \lim_{\epsilon \to 0} \int_{E_x}\int_{E_\xi} v_\veps(x,\xi) \cdot \psi(x,\xi,x/\epsilon)\dd \xi\dd x 
        = \int_{E_x}\int_{E_\xi}\int_{Q_x}\left(\int_{Q_\xi} \tilde{v}(x,\xi,z,\omega)\dl \omega\right) \cdot \psi(x,\xi,z)\dd z\dd \xi \dd x.
    \end{equation*}
    This shows the weak $x$-two-scale convergence \eqref{eq:w_xts} with \(\inner{\tilde{v}}_{Q_\xi}\) as the limit function.
    
    Conversely, let $ v_\veps \wxtwoscale v$ as $\veps \to 0$ for some \(v\in L^2(E_x\times E_\xi\times Q_x;\R^m)\). Since any weakly \(x\)-two-scale converging sequence is bounded in $L^2(E;\R^m)$ (cf., e.g., in \cite[Theorem~6]{LNW02}), by Property \ref{P1} we find a (non-relabeled) subsequence and an element \(\tilde{v}\in L^2(E\times Q_\eta;\R^m)\) such that $v_\epsilon \wtwoscale \tilde{v}$ as $\veps \to 0$. It follows now from the first implication and the uniqueness of the \(x\)-two-scale limit that \(\inner{\tilde{v}}_{Q_\xi}=v\) almost everywhere in \(E_\xi\times E_x\times Q_x\). \smallskip

    \textit{Part (ii).}
    Let $(v_\epsilon)_{\epsilon}$ strongly two-scale converge to $\tilde{v}\in L^2(E\times Q_\eta;\R^m)$, where $\tilde{v}$ is constant in the last variable. We denote by $v \in L^2(E_x\times E_\xi \times Q_x;\R^m)$ the function obtained from $\tilde{v}$ by removing the dependence on the last variable. By Part~\ref{it:ts_to_xts} we deduce that $(v_\epsilon)_{\epsilon}$ weakly $x$-two-scale converges to $v$. Moreover, using the strong two-scale convergence, we find
    \[
    \lim_{\epsilon \to 0} \norm{v_\epsilon}_{L^2(E)} = \norm{\tilde{v}}_{L^2(E\times Q_\eta)} = \norm{v}_{L^2(E_x \times E_\xi\times Q_x)},
    \]
    thereby establishing \eqref{eq:s_xts}. 
    
    Conversely, if \(\sequence{v}{\veps}\) strongly \(x\)-two-scale converges to \(v\), then we find by Part~\ref{it:ts_to_xts} that (up to a subsequence) $(v_\epsilon)_\epsilon$ weakly two-scale converges to some $\tilde{v} \in L^2(E\times Q_\eta;\R^m)$ satisfying the identity \(\inner{\tilde{v}}_{Q_\xi}=v\). Utilizing Jensen's inequality and \eqref{eq:norm_lsc}, we obtain that
    \begin{align}
    \begin{split}\label{eq:sts_jensen}
         \norm{v}_{L^2(E_\xi\times E_x\times Q_x)}^2
        &= \int_{E_x}\int_{E_\xi}\int_{Q_x} |v(x,\xi,z)|^2\dl z\dl \xi\dl x\\
        &\le \norm{\tilde{v}}_{L^2(E\times Q)}^2
        \le \liminf_{\veps \to 0} \norm{v_\veps}_{L^2(E)}^2
        = \norm{v}_{L^2(E_\xi\times E_x\times Q_x)}^2.
    \end{split}
    \end{align}
    Since the leftmost and rightmost term in \eqref{eq:sts_jensen} coincide, all the inequalities must be equalities. In particular, equality in Jensen's inequality implies that $\tilde{v}$ is constant in the last variable and, thus, agrees with $v$. This establishes the strong two-scale convergence of $(v_\epsilon)_\epsilon$ to $v$ as in Definition~\ref{def:ts}~\ref{it:sts}. Moreover, as the limit is independent of the chosen subsequence, the convergence also holds for the entire sequence by Urysohn's subsequence principle.
\end{proof}

We now state some properties about $x$-two-scale convergence that will be useful for the rest of the paper.
The following lemma gives a \(x\)-two-scale compactness result, which is a direct consequence of classical two-scale compactness (see \ref{P1}) and Lemma~\ref{lem:ts_and_xts}.
\begin{lemma}\label{lem:xts_comp}
    For each bounded sequence \(\sequence{v}{\veps}\subset L^2(E;\R^m)\) there exists a subsequence and some \(v\in L^2(E_\xi\times E_x\times Q_x,\R^m)\) such that the subsequence weakly \(x\)-two-scale converges to \(v\).
\end{lemma}

We now provide a suitable \(x\)-two-scale variant of the lower semicontinuity result stated in \ref{P2}. 
\begin{lemma}\label{lem:x_norm_lsc}
    Let \(\sequence{v}{\veps}\subset L^2(E;\R^m)\) and \(v\in L^2(E_\xi\times E_x\times Q_x;\R^m)\) such that \(\sequence{v}{\veps}\) weakly \(x\)-two-scale converges to \(v\). Then
    \begin{equation*}
        \liminf_{\veps\to 0}\norm{v_\veps}_{L^2(E)}
        \ge \norm{v}_{L^2(E_\xi\times E_x\times Q_x)}.
    \end{equation*}
\end{lemma}
\begin{proof} Passing to subsequences, the thesis follows directly from Lemma~\ref{lem:ts_and_xts}\,\ref{it:ts_to_xts} and the application of Jensen's inequality, similarly to \eqref{eq:sts_jensen}.
    % We present here instead a more direct strategy following \cite[Theorem~10]{LNW02}. We choose a sequence \(\seq{\psi}\subset L^2(E;C_{\#}(Q_x;\R^m))\) which approximates \(v\) in \(L^2(E_\xi\times E_x\times Q_x;\R^m)\). Then, using Young's inequality, we have
    % \begin{align*}
    %     \iint_{E} |v_\veps(x,\xi)|^2\dl x\dl \xi
    %     \ge 2\iint_{E}v_\veps(x,\xi)\cdot \psi_k(x,\xi,x/\veps)\dl x\dl \xi
    %     -\iint_{E} |\psi_k(x,\xi,x/\veps)|^2\dl x\dl \xi.
    % \end{align*}
    % By \eqref{eq:w_xts} and the aforementioned strong two-scale behavior of \(\psi_k(x,\xi,x/\veps)\), the right-hand side converges in the limit \(\veps\to 0\) to
    % \begin{equation*}
    %     2\iiint_{E_\xi\times E_x\times Q_x}v(x,\xi,z)\cdot \psi_k(x,\xi,z)\dl z\dl x\dl \xi
    %     -\iiint_{E_\xi\times E_x\times Q_x} |\psi_k(x,\xi,z)|^2\dl x\dl \xi.
    % \end{equation*}
    % Performing now the limit passage \(k\to 0\) we find indeed
    % \begin{equation*}
    %     \liminf_{\veps\to 0} \iint_{E} |v_\veps(x,\xi)|^2\dl x\dl \xi
    %     \ge \iiint_{E_\xi\times E_x\times Q_x} |v(x,\xi,z)|^2\dl z\dl x\dl \xi,
    % \end{equation*}
    % concluding the proof.
\end{proof} 

We also show that for bounded sequences the space of test functions in \eqref{eq:w_xts} might be replaced by any dense subset of \(L^2(E;C_{\#}(Q_x;\R^m))\), such as, e.g., smooth test functions in \(C_c^{\infty}(E;C_{\#}^{\infty}(Q_x;\R^m))\).
\begin{lemma}\label{lem:x_ts_approx}
    Let \(\sequence{v}{\veps}\subset L^2(E;\R^m)\) be a bounded sequence and \(v\in L^2(E_\xi\times E_x\times Q_x;\R^m)\). Let \(D\subset L^2(E;C_{\#}(Q_x;\R^m))\) be a dense subset of functions in the respective topology. If 
    \begin{equation}\label{eq:x_ts_approx}
    \lim_{\veps \to 0}\int_{E_x}\int_{E_\xi} v_\epsilon(x,\xi) \cdot \psi(x,\xi,x/\epsilon)\dd \xi\dd x 
    = \int_{E_x}\int_{E_\xi}\int_{Q_x}v(x,\xi,z) \cdot \psi(x,\xi,z)\dd z\dd \xi \dd x
    \end{equation}
    for all \(\psi\in D\), then \(\sequence{v}{\veps}\) weakly \(x\)-two-scale converges to \(v\).
\end{lemma}
\begin{proof}
By Lemma~\ref{lem:xts_comp} we know that, up to a non-relabeled subsequence, $(v_\epsilon)_\epsilon$ weakly $x$-two-scale converges to some $v'\in L^2(E_\xi\times E_x\times Q_x;\R^m)$. Then, by \eqref{eq:x_ts_approx}, we deduce that
\[
\int_{E_x}\int_{E_\xi}\int_{Q_x}v(x,\xi,z) \cdot \psi(x,\xi,z)\dd z\dd \xi \dd x = \int_{E_x}\int_{E_\xi}\int_{Q_x}v'(x,\xi,z)\cdot \psi(x,\xi,z)\dd z\dd \xi \dd x,
\]
for all $\psi \in D$. By density of $D$, we infer that $v=v'$. Hence, up to a subsequence, $(v_\epsilon)_\epsilon$ weakly $x$-two-scale converges to $v$. Since this limit is independent of the chosen subsequence, the convergence holds for the entire sequence.
\end{proof}

Next, we treat oscillating functions of the form \( \psi(x,\xi,x/\veps)\) with \(\psi \in L^2(E_\xi\times Q_x;C(\overline{E_x};\R^m))\). Due to the continuity in the variable $x$, which is satisfied for the homogenizing coefficients later on, 
Lemma~\ref{lem:ts_and_xts} also guarantees the strong \(x\)-two-scale convergence of \( \psi(x,\xi,x/\veps)\) to \( \psi(x,\xi,z)\) as $\veps \to 0$. % (once again as in \cite{Vi06} by using the unfolding operator). Hence, Lemma~\ref{lem:ts_and_xts} entails that the sequence \(\sequence{v}{\veps}\) also converges strongly \(x\)-two-scale to \(\psi(x,\xi,z)\). The same holds when there is only continuity in $x$, which is needed for the homogenizing coefficients later.
\begin{lemma}\label{le:sxts_nocontinuity}
Let $E_x$ be bounded and $\psi \in L^2(E_\xi\times Q_x;C(\overline{E_x};\R^m))$. Then, the sequence $(\psi_\epsilon)_\epsilon$ defined by $\psi_\epsilon (x,\xi):=\psi(x,\xi,x/\epsilon)$ satisfies
\[
\psi_\epsilon(x,\xi) \xtwoscale \psi(x,\xi,z) \quad \text{in $L^2(E;\R^m)$}.
\]
\end{lemma}
\begin{proof}
    Define the function
    \[
    \tilde{\psi}(x,z):=\int_{E_\xi}\abs{\psi(x,\xi,z)}^2\dd \xi \quad \text{for all $x \in E_x$ and a.e.~$z \in Q_x$,}
    \]
    then it is straightforward to check that $\tilde{\psi} \in L^1_{\#}(Q_x;C(\overline{E_x};\R^m))$. To be precise, we have
    \begin{align*}
        \norm{\tilde{\psi}}_{L^1_{\#}(Q_x;C(\overline{E_x}))}
        &= \int _{Q_x}\sup_{\overline{E_x}}|\tilde{\psi}(\cdot,z)|\dl z\\
        &\le \int _{Q_x}\int_{E_\xi}\sup_{\overline{E_x}}|\psi(\cdot,\xi,z)|^2 \dl \xi\dl z
        = \norm{\psi}_{L^2(E_\xi\times Q_x;C(\overline{E_x}))}.
    \end{align*}
    Hence, by \cite[Corollary~5.4]{All92} (see also \cite[Theorem~2, Theorem~3]{LNW02}) we find that
    \begin{align*}
        \lim_{\epsilon \to 0}\norm{\psi_\epsilon}_{L^2(E)}^2 &= \lim_{\epsilon \to 0} \int_{E_x}\int_{E_\xi} \abs{\psi(x,\xi,x/\epsilon)}^2\dd \xi \dd x = \lim_{\epsilon \to 0}\int_{E_x} \tilde{\psi}(x,x/\epsilon)\dd x \\
        &= \int_{E_x}\int_{Q_x}\tilde{\psi}(x,z)\dd z\dd x = \int_{E_x}\int_{E_\xi}\int_{Q_x}\abs{\psi(x,\xi,z)}^2\dd z\dd \xi\dd x.
    \end{align*}
    Therefore, it remains to prove the weak convergence $\psi_\epsilon \wxtwoscale \psi$. This can be argued in a similar way, since for every $\varphi \in C_c^{\infty}(E;C_{\#}^{\infty}(Q_x;\R^m))$, the function $\psi_{\varphi}$, defined by
    \[
    \psi_{\varphi}(x,z):=\int_{E_\xi}\psi(x,\xi,z) \cdot \varphi(x,\xi,z)\dd \xi \quad \text{for all $x \in E_x$ and a.e.~$z \in Q_x$,}
    \]
    lies in $ L^1_{\#}(Q_x;C(\overline{E_x};\R^m))$. In the same way, we obtain that
    \begin{align*}
        &\lim_{\epsilon \to 0} \int_{E_x}\int_{E_\xi} \psi(x,\xi,x/\epsilon)\cdot \varphi(x,\xi,x/\epsilon)\dd \xi \dd x = \lim_{\epsilon \to 0}\int_{E_x} \psi_{\varphi}(x,x/\epsilon)\dd x \\
        &\quad= \int_{E_x}\int_{Q_x}\psi_{\varphi}(x,z)\dd z\dd x = \int_{E_x}\int_{E_\xi}\int_{Q_x}\psi(x,\xi,z)\cdot \varphi(x,\xi,z)\dd z\dd \xi\dd x,
    \end{align*}
    from which via Lemma~\ref{lem:x_ts_approx} we conclude that $\psi_\epsilon \wxtwoscale \psi$.
\end{proof}

Finally, we observe that Property \ref{P3} also holds true in the context of \(x\)-two-scale convergence. The following lemma collects this result along with two additional properties that will be useful later.
\begin{lemma}\label{lem:xts_products}
    Let \(\sequence{u}{\veps},\sequence{v}{\veps}\subset L^2(E;\R^m)\) be two sequences such that $u_\veps \wxtwoscale u$ and $v_\veps \xtwoscale v$ as $\veps \to 0$, for some \(u,v\in L^2(E_\xi\times E_x\times Q_x;\R^m)\). Then:
    \begin{enumerate}[label=(\roman*)]
        \item \label{xts_products-i)}
            There holds
            \begin{equation}\label{eq:xts_products}
            \lim_{\veps\to 0}\int_{E_x}\int_{E_\xi} u_\veps(x,\xi) \cdot v_\veps(x,\xi)\dl \xi\dl x
                 =\int_{E_x}\int_{E_\xi}\int_{Q_x} u(x,\xi,z)\cdot v(x,\xi,z)\dl z\dl \xi\dl x.
            \end{equation}
        \item \label{xts_products-ii)}
            If $(v_\epsilon)_\epsilon$ is bounded in $L^{\infty}(E;\R^m)$, then $(u_\epsilon \cdot v_\epsilon)_\epsilon \wxtwoscale u \cdot v$ in \(L^2(E)\).
        \item \label{xts_products-iii)}
            If also \(\sequence{u}{\veps}\) strongly \(x\)-two-scale converges to \(u\), then $(u_\epsilon \cdot v_\epsilon)_\epsilon \xtwoscale u \cdot v$ in \(L^2(E)\).
    \end{enumerate}

\end{lemma}
\begin{proof}
\textit{Part (i).}
%It is enough to prove both statements up to subsequence. 
By Lemma~\ref{lem:ts_and_xts}\,\ref{it:ts_to_xts}, we find that, up to a non-relabeled subsequence, there exists some $\tilde{u} \in L^2(E\times Q_\eta;\R^m)$ with $\inner{\tilde{u}}_{Q_\xi}=u$ such that %$(u_\epsilon)_\epsilon$ weakly two-scale converges to $\tilde{u}$. 
$u_\veps \wtwoscale \tilde{u}$ in $L^2(E; \R^m)$ as $ \veps \to 0$.
Moreover, we have that $v_\veps \twoscale v$ in $L^2(E; \R^m)$ as $\veps\to 0.$ %strongly two-scale converges to $v$. 
Thus, using \ref{P3}, we infer that
\begin{align*}
    \lim_{\epsilon \to 0} \int_{E_x}\int_{E_\xi} u_\veps(x,\xi) \cdot v_\veps(x,\xi)\dl \xi\dl x
         &=\int_{E_x}\int_{E_\xi}\int_{Q_x} \left( \int_{Q_\xi}\tilde{u}(x,\xi,z,\omega) \dl \omega \right)\cdot v(x,\xi,z)\dl z\dl \xi\dl x\\
         &=\int_{E_x}\int_{E_\xi}\int_{Q_x} u(x,\xi,z)\cdot v(x,\xi,z)\dl z\dl \xi\dl x.  
\end{align*}
%which proves \eqref{eq:xts_products}.
\smallskip

\textit{Part (ii).} 
% Since $(v_\epsilon)_\epsilon$ is bounded in $L^{\infty}(E;\R^m)$ and $u_\veps \wxtwoscale u$ in $L^2(E; \R^m)$, we find that the sequence $(f_\veps)_\veps := (u_\epsilon \cdot v_\epsilon)_\epsilon$ is bounded in $L^2(E; \R)$. By Lemma \ref{lem:xts_comp}, $f_\veps \wxtwoscale f$ as $\veps \to 0$, for some $f \in L^2(E_\xi\times E_x\times Q_x;\R).$ Using \textit{Part (i)} \ROSS{don't know how to conclude}
Since $(v_\epsilon)_\epsilon$ is bounded in $L^{\infty}(E;\R^m)$ and $u_\veps \wxtwoscale u$ in $L^2(E; \R^m)$, we find that the sequence $(u_\epsilon \cdot v_\epsilon)_\epsilon$ is bounded in $L^2(E)$. Using that $u_\epsilon \wtwoscale \tilde{u}$ (see Part~\ref{xts_products-i)}) and $v_\epsilon \twoscale v$ in $L^2(E; \R^m)$, via periodic unfolding (cf.\ \cite{Vi06}) and by the fact that an analog version of statement \ref{xts_products-ii)} holds true in the standard Lebesgue spaces over \(\R^n\times Q_\eta\), we deduce
that $(u_\epsilon \cdot v_\epsilon)_\epsilon$ weakly two-scale converges to $\tilde{u} \cdot v$ in $L^2(E)$. Hence, by Lemma~\ref{lem:ts_and_xts}\,\ref{it:ts_to_xts}, it follows that $(u_\epsilon \cdot v_\epsilon)_\epsilon$ weakly $x$-two-scale converges to $\inner{\tilde{u}\cdot v}_{Q_\xi}=u \cdot v$.\smallskip
%Moreover, because $u_\epsilon \wtwoscale \tilde{u}$ and $v_\epsilon \twoscale v$ in $L^2(E; \R^m)$, we can apply \cite[Theorem~6.2]{Fra11} \hidde{or any better reference}\comment{I have to say that I don't know any better reference, where this is stated clearly. Maybe at this point we don't get around mentioning the unfolding operator...} to deduce that $(u_\epsilon \cdot v_\epsilon)_\epsilon$ weakly two-scale converges to $\tilde{u} \cdot v$ in $L^2(E;\R)$. Hence, by Lemma~\ref{lem:ts_and_xts}, it follows that $(u_\epsilon \cdot v_\epsilon)_\epsilon$ weakly $x$-two-scale converges to $\inner{\tilde{u}\cdot v}_{Q_\xi}=u \cdot v$.\smallskip

% For the second statement, we find by properties of two-scale convergence that $(u_\epsilon \cdot v_\epsilon)_\epsilon$ weakly two-scale converges to $\tilde{u} \cdot v$. \comment{maybe some more details} Hence, by Lemma~\ref{lem:ts_and_xts}, it follows that $(u_\epsilon \cdot v_\epsilon)_\epsilon$ weakly $x$-two-scale converges to $\inner{\tilde{u}\cdot v}_{Q_\xi}=u \cdot v$.\smallskip

\textit{Part (iii).} By Lemma~\ref{lem:ts_and_xts}\,\ref{it:xts_to_ts} we additionally know that $u_\epsilon \twoscale u$ in $L^2(E; \R^m)$. Applying periodic unfolding as in Part~\ref{xts_products-ii)}, this again reduces the statement to the setting of standard Lebesgue spaces over \(\R^n\times Q_\eta\), where the result is already known to hold. 
%so we may again apply \cite[Theorem~6.2]{Fra11} \hidde{not quite applicable, maybe just mention that it follows via unfolding} 
We then find that $(u_\epsilon \cdot v_\epsilon)_\epsilon$ strongly two-scale converges to $u \cdot v$ in $L^2(E)$. Since the limit does not depend on the $\omega$-variable, we deduce from Lemma~\ref{lem:ts_and_xts}\,\ref{it:xts_to_ts} that $(u_\epsilon \cdot v_\epsilon)_\epsilon$ also strongly $x$-two-scale converges to $u \cdot v$.

\end{proof}

%%%%%%%%%%%%%%%%%%%%%%%%%%%%%%%%%%%%%%%%%%%%%%%%%%

%%%%%%%%%%%%%%%%%%%%%%%%%%%%%%%%%%%%%%%%%%%%%%%%%%%%%%%%%%%%%%%%%%%%%%%%%%%%%%%%%%%%%%%%%%%

\section{Compactness and nonlocal two-scale limits}\label{sec:compactness}

In this section we present the compactness results related to the nonlocal functionals. While the strong $L^2$-compactness result follows readily from the analysis in \cite{DDG24}, the identification of the two-scale limits of the nonlocal difference quotients is completely new. 

We first prove the following auxiliary lemma, which utilizes some results on unbounded linear operators, see e.g.~\cite[Chapter~2]{Bre11} for an introduction to this topic.
\begin{lemma}\label{le:Aoperator}
    Let $S_\rho:\Dom(S_\rho) \subset L^2(\Omega ;L^2_{\#}(Q;\R^3)) \to L^2(\Omega\times\R^3;L^2_{\#}(Q;\R^3))$ be the operator given by
    \[
    S_\rho(w)(x,\xi,z):= \rho(\xi)^{1/2}\frac{(w(x,z+\xi)-w(x,z))}{|\xi|} \quad \text{for a.e.~$(x,\xi,z) \in \Omega\times \R^3\times Q$}.
    \]
    Then, $S_\rho$ is a densely defined unbounded linear operator with closed range, and its adjoint is given by $S_\rho^*:\Dom(S^*_\rho) \subset L^2(\Omega\times\R^3;L^2_{\#}(Q;\R^3)) \to L^2(\Omega ;L^2_{\#}(Q;\R^3))$ with
    \begin{equation}\label{eq:Aadjoint}
    S^*_\rho(u)(x,z)=\int_{\R^3}\rho(\xi)^{1/2}\frac{(u(x,\xi,z-\xi)-u(x,\xi,z))}{|\xi|}\dd \xi  \quad \text{for a.e.~$(x,z) \in \Omega \times Q$}.
    \end{equation}
\end{lemma}
\begin{proof}
    For shorter notation, we denote the Hilbert spaces $H_1:=L^2(\Omega ;L^2_{\#}(Q;\R^3))$ and $H_2:=L^2(\Omega\times\R^3;L^2_{\#}(Q;\R^3))$. Therefore, $S_\rho : \Dom(S_\rho) \subset H_1 \to H_2.$ It can be checked that there holds $C_c^{\infty}(\Omega;C^\infty_{\#}(Q;\R^3)) \subset \Dom(S_\rho)$, which shows that $S_\rho$ is densely defined. Moreover, if $u \in \Dom(S^*_\rho)$, i.e., $\Dom(S_\rho) \ni w \mapsto \inner{S_\rho(w),u}_{H_2}$ is a bounded linear functional, then for all $w \in \Dom(S_\rho)$ we compute that
    \begin{align*}
        \inner{S_\rho(w),u}_{H_2} &= \int_{\Omega}\int_{\R^3}\int_{Q} \rho(\xi)^{1/2}\frac{(w(x,z+\xi)-w(x,z))}{|\xi|}\cdot u(x,\xi,z)\dd z\dd \xi \dd x\\
        &= \int_{\Omega}\int_{\R^3}\int_{Q} w(x,z) \cdot \left(\rho(\xi)^{1/2}\frac{(u(x,\xi,z-\xi)-u(x,\xi,z))}{|\xi|}\right)\dd z\dd \xi \dd x \\
        &=\int_{\Omega}\int_{Q} w(x,z) \cdot\left(\int_{\R^3}\rho(\xi)^{1/2}\frac{(u(x,\xi,z-\xi)-u(x,\xi,z))}{|\xi|}\dd \xi \right) \dd z\dd x,
    \end{align*}
    using the change of variables $z + \xi \mapsto z$ with fixed $\xi \in \R^3$ and Fubini's theorem. This shows that $S_\rho^*(u)$ is indeed given by \eqref{eq:Aadjoint}.
    
    Finally, to prove that the range of $S_\rho$ is closed, we consider a sequence $(w_j)_j \subset \Dom(S_\rho)$ such that $S_\rho(w_j) \to F$ in $H_2$ as $j \to \infty$. Moreover, we may assume that $\int_{Q}w_j(x,z)\dd z=0$ for a.e.~$x \in \Omega$, since this does not affect the value of $S_\rho$. 
    % Now, due to the lower bound on $\rho$ from \eqref{eq:rhocoercive}, we may use the nonlocal Poincar\'{e}-Wirtinger inequality \comment{which I state now already in \eqref{eq:nonlocal_poincare}}, see e.g.~\cite[Proposition~4.2]{BeM14} or \cite[Proposition~4.2]{AAB23}, to find
    % \begin{align*}
    % \norm{w_j}^2_{H_1}&=\iint_{\Omega \times Q}\abs*{w_j(x,z)}^2\dd z\dd x \\
    % &\leq C \iiint_{\Omega\times Q \times Q}\rho(y-z)\frac{\abs*{w_j(x,y)-w_j(x,z)}^2}{\abs{y-z}^2}\dd z\dd y\dd x \\
    % &\leq C \iiint_{\Omega\times \R^3 \times Q}\rho(\xi)\frac{\abs*{w_j(x,z+\xi)-w_j(x,z)}^2}{\abs{\xi}^2}\dd z\dd \xi\dd x = C\norm*{S(w_j)}^2_{H_2}.
    % \end{align*}
    Due to inequality \eqref{eq:nonlocal_poincare} in Appendix~\ref{app:a}, we find that there is a $C>0$ such that
    \[
    \norm{w_j}^2_{H_1} \leq C\norm*{S_\rho(w_j)}^2_{H_2} \quad \text{for all $j \in \N$}.
    \]
    Therefore, up to a non-relabeled subsequence, we find that $w_j \rightharpoonup w$ in $H_1$. For all $u \in \Dom(S_\rho^*)$, we now find that
    \[
    \inner{F,u}_{H_2}=\lim_{j \to \infty} \inner{S_\rho(w_j),u}_{H_2}=\lim_{j \to \infty} \inner{w_j, S_\rho^*(u)}_{H_1} = \inner{w,S_\rho^*(u)}_{H_1}.
    \]
    This shows that $F=S_\rho(w)$, which proves that the range of $S_\rho$ is closed.
\end{proof}
We also need a density result for functions in the kernel of $S_\rho^*$.
\begin{lemma}\label{le:densitykernel}
    Let $\psi \in L^2(\Omega\times\R^3;L^2_{\#}(Q;\R^3))$ be such that $S_\rho^*(\psi)=0$. Then, there exists a sequence $\seq{\psi} \subset L^2(\Omega\times\R^3;L^2_{\#}(Q;\R^3))$ converging to $\psi$ in $L^2(\Omega\times\R^3;L^2_{\#}(Q;\R^3))$ such that $S_\rho^*(\psi_k)=0$ and
    \begin{equation}\label{eq:smoothpsik}
        \supp \psi_k(\cdot,\xi,z) \subset K_k \Subset \Omega \ \ \text{for a.e.~$(\xi,z) \in \R^3 \times Q$} \quad \text{and} \quad \int_{\R^3}\norm{\psi_k(\cdot,\xi,\cdot)}_{C^2(\Omega \times Q)}^2\dd \xi <+ \infty,
    \end{equation}
    for all $k \in \N$.
\end{lemma}
\begin{proof}
      Let $\varphi$ be a standard 
     mollifier, i.e., \(\varphi\in C_c^\infty(\R^3\times \R^3; [0,+\infty))\), with \(\supp \varphi \subset B_1(0)\) and \(\norm{\varphi}_{L^1(\R^3\times \R^3)}=1\). Since we will mollify simultaneously in the \(x\) and \(z\) arguments, we use the notation \(y=(y_1,y_2)\in \R^3\times\R^3\) for the mollifying variable. For $k \in \N$, we define \(\varphi_k(y)\coloneqq k^{6}\varphi(ky)\). Moreover, we choose an associated sequence of cut-off functions \(\seq{\chi}\subset C_c^\infty(\Om;[0,1])\) such that \(\chi_k= 1\) on \(\Om_{3/k}\) and \(\chi_k=0\) on \(\Om\setminus \Om_{2/k}\) for all $k \in \N$. Then, we define the sequence
     \[
     \psi_k(x,\xi,z):=\chi_k(x) (\varphi_k * \psi(\cdot,\xi,\cdot))(x,z) =  \chi_k(x)\int_{\R^6}\varphi_{k}(y)\bar{\psi}(x-y_1,\xi,z-y_2)\dl y \quad \text{for $k \in \N$,}
     \]
     where \(\bar{\psi}\) denotes the extension in \(x\) of \(\psi\) to the outside of \(\Om\) by zero. The first property in \eqref{eq:smoothpsik} is clear by the definition of $\varphi_{k}$ and $\chi_k$, while the second one is readily verified via Young's convolution inequality. Furthermore, by standard properties of mollification we know that $\psi_k \to \psi$ in $L^2(\Omega\times\R^3;L^2_{\#}(Q;\R^3))$ as $k \to \infty$. It remains to show that $S_\rho^*(\psi_k)=0$ for all $k \in \N$, which can be done via a direct computation. Indeed, we have that
     \begin{align*}
         S_\rho^*(\psi_k)(x,z)&=\int_{\R^3}\rho(\xi)^{1/2}\frac{(\psi_k(x,\xi,z-\xi)-\psi_k(x,\xi,z))}{\abs{\xi}}\dl \xi \\
         &= \chi_k(x) \int_{\R^6} \varphi_k(y)\int_{\R^3}\rho(\xi)^{1/2}\frac{(\bar{\psi}(x-y_1,\xi,z-\xi-y_2)-\bar{\psi}(x-y_1,\xi,z-y_2))}{\abs{\xi}}\dd \xi\dl y \\
         &=\chi_k(x) \int_{\R^6} \varphi_k(y) S_\rho^{*}(\bar{\psi})(x-y_1,z-y_2)\dl y =0,
     \end{align*}
     where we have used that $S_\rho^*(\psi)=0$ in the last equality.
\end{proof}
We now state the main result of this section, which will play a crucial role in the proof of the liminf-inequality. It consists of a compactness result with a suitable $x$-two-scale convergence result in this setting. %After performing the change of variable as in \eqref{Feps-shifted}-\eqref{Heps-shifted}, as already alluded to in Section~\ref{sec:prelim}, we will restrict to test functions only taking into account oscillations in $x$, that is, functions of the form $\psi(x,\xi,x/\veps)$. 

For $m \in L^2(\Omega;\R^3)$, we introduce the auxiliary function
\begin{equation}\label{eq:vepsform}
\Delta^{\rho}_\epsilon [m](x,\xi):=\1_{\Omega_{\epsilon\xi}}(x)\rho(\xi)^{1/2}\frac{(m(x+\epsilon\xi)-m(x))}{\epsilon|\xi|},
\end{equation}
defined for a.e.\ $(x,\xi) \in \Omega \times \R^3$. We have then the following result, of which Part~\ref{it:2scaleii} and Part~\ref{it:2scaleiii} can be seen as nonlocal analogs of the two-scale decomposition in \cite[Proposition~1.14~(i)]{All92} and \cite[Proposition~2.2]{DaD20}.
%and
%\begin{equation}
%\Delta^{\nu}_\epsilon[m](x,\xi):=\1_{\Omega_{\epsilon\xi}}(x)\nu(\xi)\cdot \frac{(m(x+\epsilon\xi)\times m(x))}{\epsilon|\xi|},
%\end{equation} 
\begin{theorem}[Equi-coercivity and two-scale compactness]\label{th:2scale}
    Let $(m_\epsilon)_\epsilon \subset L^2(\Omega;\R^3)$ be a bounded sequence such that 
    \begin{equation}\label{eq:bound}
    \limsup_{\epsilon \to 0} \Ecal_\epsilon(m_\epsilon) < +\infty.
    \end{equation}
    Then, there exist a unique $m_0 \in H^1(\Omega;\R^3)$ and $w \in L^2(\Omega;L^2_{\#,0}(Q;\R^3))$ such that up to a non-relabeled subsequence:
    \begin{enumerate}[label=(\roman*)]
    \item \label{it:2scalei} The sequence $(m_\epsilon)_\epsilon$ converges to $m_0$ in $L^2(\Omega;\R^3)$.
        
        \item \label{it:2scaleii} The sequence $(\Delta_\epsilon^\rho[m_\veps])_\veps$ weakly $x$-two-scale converges in $L^2(\Omega\times\R^3;\R^3)$ to
    \begin{equation}\label{eq:vform}
    \Delta^\rho[m_0,w](x,\xi,z):=\rho(\xi)^{1/2}\left(\nabla m_0(x)\frac{\xi}{|\xi|}+\frac{w(x,z+\xi)-w(x,z)}{|\xi|}\right).
    \end{equation}
    % \item[(ii)] The sequence $\Delta_\epsilon^\nu[m_\veps]$ weakly $x$-two-scale converges in $L^1(\Omega\times\R^3;\R^3)$ to\hidde{Is $L^1$ good enough? Otherwise we should use that $\kappa(x/\epsilon,x/\epsilon+\xi) (m_\epsilon \times \nu/\rho^{1/2})$ strongly $x$-two-scale converges in $L^2$ and simply omit this part}
    % \[
    % v^\nu(x,\xi,z)= \left(\nabla m_0(x)\frac{\xi}{|\xi|}+\frac{w(x,z+\xi)-w(x,z)}{|\xi|}\right)\cdot \left( m_0(x) \times \nu(\xi)\right).
    % \]
    \item \label{it:2scaleiii} If additionally $(m_\epsilon)_\epsilon \subset L^2(\Omega;\S^2)$, then $w(x,z) \in T_{m_0(x)}\S^2$ for a.e.~$z\in Q$.
    \end{enumerate}
    % weakly $x$-two-scale converges in $L^2(\Omega\times\R^3;\R^3)$ (up to a non-relabeled subsequence) to
    % \begin{equation}\label{eq:vform}
    % v(x,\xi,z):=\rho(\xi)^{1/2}\left(\nabla m_0(x)\frac{\xi}{|\xi|}+\frac{w(x,z+\xi)-w(x,z)}{|\xi|}\right).
    % \end{equation}
\end{theorem}
\begin{proof}
%    We start recalling that by Lemma~\ref{lem:x_ts_approx} it is sufficient to test with functions $\psi \in C_c^{\infty}(\Om\times \R^3;C_{\#}^{\infty}(Q;\R^3))$ because the latter space is dense in \(L^2(\Om\times \R^3;C_{\#}(Q;\R^3))\).
\textit{Part (i).} Due to the boundedness of $(m_\epsilon)_\epsilon$ in $L^2(\Omega;\R^3)$ and Lemma~\ref{lem:antisym_by_sym}, it holds that also $\limsup_{\epsilon \to 0} \Fcal_\epsilon(m_\epsilon) < +\infty$, and therefore using that $a \geq a_0>0$ from \ref{H1}, it yields
    \begin{equation}\label{eq:referenceenergy}
    \limsup_{\epsilon \to 0} \int_{\Omega} \int_{\Omega}\frac{1}{\epsilon^3}\rho\left(\frac{y-x}{\epsilon}\right)\frac{|m_\epsilon(y)-m_\epsilon(x)|^2}{|y-x|^2}\dd x \dd y < +\infty.
    \end{equation}
  In light of Assumption \ref{H2}, $\rho$ is (up to a constant) lower bounded by the radial function $\tilde{\rho}(\xi):=\mathds{1}_{B_r(0)}(\xi)\abs{\xi}^2$ for some $r > 0$. Then, considering \eqref{eq:referenceenergy} with the kernel $\tilde{\rho}_\veps(\xi):=  c \, \veps^{-3} \, \tilde{\rho}(\xi/\veps)$ and $c = \|\tilde{\rho}\|^{-1}_{L^1(\R^3)}$, we are under the assumptions of \cite[Theorem~1.2]{Pon04}, from which we deduce the statement. \medskip

    \textit{Part (ii).} %\ (Unconstrained case)} 
    We assume without loss of generality that $\rho$ has compact support in $B_R(0)$ for some $R > 0$, since otherwise we may multiply $\rho$ by $\mathds{1}_{B_R(0)}$ and exhaust $\R^3$ by letting $R \to \infty$. 
    
    By Part~\ref{it:2scalei} we find, up to a non-relabeled subsequence, that $m_\epsilon \to m_0$ in $L^2(\Omega;\R^3)$ as $\varepsilon \to 0$ for some $m_0 \in H^1(\Omega;\R^3)$. Using \eqref{eq:bound} we can invoke the $x$-two-scale compactness in Lemma~\ref{lem:xts_comp} to infer that there exists some $v \in L^2(\R^3 \times \Omega \times Q;\R^3)$ such that, for every $\psi \in L^2(\Omega \times \R^3;C_{\#}(Q;\R^3))$,
    \begin{equation}\label{eq:twoscalelimit}
    \lim_{\epsilon \to 0} \int_{\R^3}\int_{\Omega} \Delta_\epsilon^\rho[m_\veps](x,\xi) \cdot \psi(x,\xi,x/\epsilon)\dd x\dd \xi = \int_{\R^3}\int_{\Omega}\int_{Q}v(x,\xi,z) \cdot \psi(x,\xi,z)\dd z\dd x \dd \xi.
    \end{equation}    
    %Indeed, by weak two-scale compactness there exists a function \(\tilde{v}\in L^2(\Om\times \R^3\times Q\times Q;\R^3)\) such that the analogue of \eqref{eq:w_ts} holds. This implies that for all $\psi \in L^2(\Omega \times \R^3;C_{\#}(Q;\R^3))$ we have
    %\begin{equation*}
    %    \lim_{\epsilon \to 0} \iint_{\R^3 \times \Omega} \Delta_\epsilon^\rho[m_\veps](x,\xi) \psi(x,\xi,x/\epsilon)\dd x\dd \xi 
    %    = \iiint_{\R^3\times\Omega\times Q}\left(\int_Q \tilde{v}(x,\xi,z,\eta)\dl \eta\right)\psi(x,\xi,z)\dd z\dd x \dd \xi.
    %\end{equation*}
    %Setting \(v\coloneqq \int_Q \tilde{v}(x,\xi,z,\eta)\dl \eta\) we have \eqref{eq:twoscalelimit}.
    It remains to prove that $v$ has the form as in \eqref{eq:vform}. We then consider a test function $\psi$ that has compact support in the first argument and that is smooth in the first and third argument, namely 
    \begin{equation}\label{eq:smoothpsi}
        \supp \psi(\cdot,\xi,z) \subset K %\Subset \Omega 
        \ \ \text{for a.e.~$(\xi,z) \in \R^3 \times Q$} \quad \text{and} \quad \int_{\R^3}\norm{\psi(\cdot,\xi,\cdot)}_{C^2(\Omega \times Q)}^2\dd \xi <+ \infty,
    \end{equation}
    for some $K \Subset \Omega$. 
    We then choose $R>0$ large enough such that $K \times \supp \rho \subset \Omega_{1/R} \times B_R(0)$, where we remind that $\Omega_{1/R}  = \{ x \in \Omega : \operatorname{dist}(x, \partial \Om) > 1/R\} $. Thus, for every $\epsilon <1/R^2$
    \begin{equation}\label{eq:mid-term}
    \int_{\R^3}\int_{\Omega} \Delta_\epsilon^\rho[m_\veps](x,\xi) \cdot \psi(x,\xi,x/\epsilon)\dd x\dd \xi = \int_{\R^3}\int_{\R^3} \rho(\xi)^{1/2}\frac{(m_\epsilon(x+\epsilon\xi)-m_\epsilon(x))}{\epsilon|\xi|} \cdot \psi(x,\xi,x/\epsilon)\dd x\dd \xi,
    \end{equation}
    where we extended all functions $m_\veps$ as zero outside $\Omega$. For such $\epsilon > 0$, via the change of variables $x \mapsto x - \veps \xi$ for fixed $\xi \in \R^3$, we find that
    \begin{align}
        &\int_{\R^3}\int_{\R^3} \rho(\xi)^{1/2}\frac{(m_\epsilon(x+\epsilon\xi)-m_\epsilon(x))}{\epsilon|\xi|} \cdot \psi(x,\xi,x/\epsilon)\dd x\dd \xi \nonumber \\
        = &\int_{\R^3}\int_{\R^3}m_\epsilon(x) \cdot \left(\rho(\xi)^{1/2}\frac{(\psi(x-\epsilon\xi,\xi,x/\epsilon-\xi)-\psi(x,\xi,x/\epsilon))}{\epsilon|\xi|}\right)\dd x\dd \xi \label{eq:VT-1}\\
        = & \int_{\R^3}\int_{\R^3}m_\epsilon(x) \cdot\left(\rho(\xi)^{1/2}\frac{(\psi(x-\epsilon\xi,\xi,x/\epsilon-\xi)-\psi(x,\xi,x/\epsilon-\xi))}{\epsilon|\xi|}\right)\dd x\dd \xi \label{eq:VT-2} \\
        &\qquad+\int_{\R^3}\int_{\R^3}m_\epsilon(x) \cdot\left(\rho(\xi)^{1/2}\frac{(\psi(x,\xi,x/\epsilon-\xi)-\psi(x,\xi,x/\epsilon))}{\epsilon|\xi|}\right)\dd x\dd \xi, \nonumber
    \end{align}
    where in the last equality \eqref{eq:VT-2} we added and subtracted the quantity $\psi(x,\xi,x/\epsilon-\xi)$. 
    In particular, choosing $\psi$ such that $S_\rho^*(\psi)=0$ (cf.~\eqref{eq:Aadjoint}), we have that the second integral in \eqref{eq:VT-2} vanishes. 
    
    Hence, by the strong $x$-two-scale convergence of $m_\epsilon \to m_0$ and the smoothness of $\psi$ in \eqref{eq:smoothpsi}, we are able to show that 
    \begin{equation}\label{1 - limit}
        \begin{split}
    \lim_{\epsilon \to 0} &\int_{\R^3}\int_{\Omega} \Delta_\epsilon^\rho[m_\veps](x,\xi) \cdot \psi(x,\xi,x/\epsilon)\dd x\dd \xi \\
    \qquad &= \int_{\R^3}\int_{\R^3}\int_{Q}m_0(x) \cdot\left(\rho(\xi)^{1/2}\nabla_x \psi(x,\xi,z-\xi)\cdot\frac{-\xi}{|\xi|}\right)\dd z\dd x\dd \xi, 
    \end{split}
    \end{equation}
    where we use %Lemma \ref{le:sxts_nocontinuity}, 
    Lemma \ref{lem:xts_products} and the fact that \begin{equation*}
        \rho(\xi)^{1/2}\frac{(\psi(x-\epsilon\xi,\xi,x/\epsilon-\xi)-\psi(x,\xi,x/\epsilon))}{\epsilon|\xi|} \, \wxtwoscale \,\rho(\xi)^{1/2}\nabla_x \psi(x,\xi,z-\xi) \quad \text{in } L^2(\Omega \times \R^3; \R^3). \quad %\text{as } \veps \to 0.
    \end{equation*} Indeed, to show this last convergence, we rely on the smoothness of $\psi$ and on the control 
    \begin{equation*}
        \left| \frac{\psi(x-\epsilon\xi,\xi,x/\epsilon-\xi)-\psi(x,\xi,x/\epsilon-\xi)}{\veps|\xi|} - \nabla_x \psi(x,\xi,x/\epsilon-\xi)\cdot\frac{-\xi}{|\xi|} \right| \leq \frac{1}{2} \veps |\xi| \, \|\nabla_x^2 \psi(\cdot,\xi,\cdot)\|^2_{C(\Omega\times Q)}
    \end{equation*} for every $\xi \in B_R(0)$, and thus, by dominated convergence theorem, we get the strong convergence to zero in $L^2(\Omega \times \R^3; \R^3).$ Additionally, by continuity of $\nabla_x \psi$ in the third variable and Lemma \ref{le:sxts_nocontinuity}, we also used that 
    \begin{equation*}
        \rho(\xi)^{1/2}\nabla_x \psi(x,\xi,x/\epsilon-\xi) \, \xtwoscale \, \rho(\xi)^{1/2}\nabla_x \psi(x,\xi,z-\xi) \quad \text{in } L^2(\Om \times \R^3; \R^3).
    \end{equation*}

    We now rewrite the limit in \eqref{1 - limit} as
    \begin{align*} &\int_{\R^3}\int_{\R^3}\int_{Q}m_0(x) \cdot \left(\rho(\xi)^{1/2}\nabla_x \psi(x,\xi,z-\xi)\frac{-\xi}{|\xi|}\right)\dd z\dd x\dd \xi %\label{1 - limit}
    \\
    \qquad = &\int_{\R^3}\int_{\R^3}\int_{Q}\rho(\xi)^{1/2}\nabla m_0(x) \frac{\xi}{|\xi|}\cdot \psi(x,\xi,z-\xi)\dd z\dd x\dd \xi %\label{2 - parts}
    \\
    \qquad = &\int_{\R^3}\int_{\R^3}\int_{Q}\rho(\xi)^{1/2}\nabla m_0(x) \frac{\xi}{|\xi|}\cdot \psi(x,\xi,z)\dd z\dd x\dd \xi, %\label{3 - cov}
    \end{align*}
    where %in \eqref{2 - parts} 
    we first used integration by parts with the fact that $\psi$ and $\rho$ have compact support, and then we applied the change of variable $z \mapsto z + \xi$ for fixed $\xi \in Q$. 
    In light of \eqref{eq:twoscalelimit}, we deduce that
    \begin{equation}\label{eq:orthogonal}
    \int_{\R^3} \int_{\Omega} \int_{Q}\left(v(x,\xi,z)-\rho(\xi)^{1/2}\nabla m_0(x) \frac{\xi}{|\xi|}\right)\cdot \psi(x,\xi,z)\dd z\dd x\dd \xi=0,
    \end{equation}
    for all $\psi$ satisfying \eqref{eq:smoothpsi} and $S_\rho^*(\psi)=0$. 
    
    Using the density result in Lemma~\ref{le:densitykernel}, we find that \eqref{eq:orthogonal} even holds for every $\psi$ in the kernel of $S_\rho^*$. By Lemma~\ref{le:Aoperator} and \cite[Theorem~2.19\,(iii)]{Bre11}, we note that the orthogonal complement of the kernel of $S_\rho^*$ coincides with the range of $S_\rho$, and so we deduce that there exists an element $w \in \Dom(S_\rho)\subset L^2(\Omega ;L^2_{\#}(Q;\R^3)) $ with average zero in the second argument such that
    \begin{equation}\label{eq:identification-limit}
    \left(v(x,\xi,z)-\rho(\xi)^{1/2}\nabla m_0(x)\cdot \frac{\xi}{|\xi|}\right)=S_\rho(w)(x,z) = \rho(\xi)^{1/2}\frac{(w(x,z+\xi)-w(x,z))}{|\xi|} 
    \end{equation}
    for a.e.~$(x,\xi,z) \in \Omega\times \R^3\times Q$. This proves that $v=\Delta^\rho[m_0,w]$ as desired.\medskip

    % \textit{Part (ii).} We can rewrite
    % \begin{align*}
    %     \Delta_\epsilon^nu[m_\veps](x,\xi)&=\rho(\xi)^{1/2} \frac{m_\epsilon(x+\epsilon\xi)-m_\epsilon(x)}{\epsilon |\xi|}\cdot \left(m_\epsilon(x) \times \frac{\nu(\xi)}{\rho(\xi)^{1/2}}\right).
    % \end{align*}
    % Hence, we can use Part~(i) to see that the first term weakly $x$-two-scale converges to \eqref{eq:vform}, while the second term strongly two-scale converges given \eqref{eq:ratiobound}. Therefore, the product weakly two-scale converges in $L^1(\Omega\times\R^3;\R^3)$ to
    % \[
    % v^\nu(x,\xi,z)=\left(\nabla m_0(x)\frac{\xi}{|\xi|}+\frac{w(x,z+\xi)-w(x,z)}{|\xi|}\right)\cdot \left( m_0(x) \times \nu(\xi)\right).
    % \]\medskip

    \textit{Part (iii).} %\ (Constrained case)} 
    For $\veps > 0$, we have that
    \begin{equation*}
     \int_{\R^3} \int_{\Om_{\veps\xi}} \rho(\xi)^{1/2}\frac{(|m_\epsilon(x+\epsilon \xi)|^2-|m_\epsilon(x)|^2)}{\epsilon|\xi|}\psi(x,\xi,x/\epsilon)\dd x \dd \xi = 0   
    \end{equation*}
    for every $\psi \in C_c^{\infty}(\Omega \times \R^3;C_{\#}^{\infty}(Q))$, since $|m_\epsilon(x+\epsilon \xi)|=|m_\epsilon(x)| = 1 $ for a.e.\ $\xi \in \R^3$ and $x \in \Om_{\veps\xi}$.
    Hence, for every $\psi \in C_c^{\infty}(\Omega \times \R^3;C_{\#}^{\infty}(Q))$, we infer that
    \begin{align}
        0 &=\lim_{\epsilon\to 0}\int_{\R^3} \int_{\Om_{\veps\xi}}\rho(\xi)^{1/2}\frac{(|m_\epsilon(x+\epsilon \xi)|^2-|m_\epsilon(x)|^2)}{\epsilon|\xi|}\psi(x,\xi,x/\epsilon)\dd x \dd \xi \nonumber  \\ 
        &=\lim_{\epsilon\to 0}\int_{\R^3} \int_{\Om_{\veps\xi}}\rho(\xi)^{1/2}\frac{(m_\epsilon(x+\epsilon \xi)-m_\epsilon(x))}{\epsilon|\xi|}\cdot(m_\epsilon(x+\epsilon\xi)+m_\epsilon(x)) \cdot\psi(x,\xi,x/\epsilon)\dd x \dd \xi \nonumber \\
        &=\lim_{\epsilon\to 0}\int_{\R^3} \int_{\Omega}\Delta_\epsilon^\rho[m_\veps](x,\xi)\cdot(m_\epsilon(x+\epsilon\xi)+m_\epsilon(x)) \cdot\psi(x,\xi,x/\epsilon)\dd x \dd \xi \label{use-delta}\\
        &=2 \int_{\R^3} \int_{\Om} \int_{Q} \frac{\rho(\xi)^{1/2}}{|\xi|}\left(\nabla m_0(x)\xi+w(x,z+\xi)-w(x,z)\right)\cdot m_0(x) \cdot \psi(x,\xi,z)\dd z\dd x\dd \xi \label{use-limit}
    \end{align}
    with $w \in L^2(\Omega;L^2_{\#,0}(Q;\R^3))$,
where in \eqref{use-delta} we used  \eqref{eq:vepsform}, while in \eqref{use-limit} we exploited the weak $x$-two-scale convergence result in $L^2(\Omega\times\R^3;\R^3)$  from Part~\ref{it:2scaleii}, together with Lemma~\ref{lem:xts_products}\,\ref{xts_products-i)}. 

Since $\nabla m_0(x) \cdot m_0(x) =0$ for a.e.~$x \in \Omega$, by \eqref{use-limit} we also find that
\[
\rho(\xi)^{1/2}\left(w(x,z+\xi)-w(x,z)\right)\cdot m_0(x)=0 \quad \text{for a.e.~$(x,\xi,z) \in \Omega \times \R^3 \times Q$,}
\]
which in turn, due to the kernel assumption \eqref{eq:rhocoercive},  yields
\[
w(x,z)\cdot m_0(x)=c \quad \text{for a.e.~$(x,z) \in \Omega\times Q$ and for some } c \in \R.
\]
Since $w$ has average zero in the second argument, we must have $c=0$ as desired. This concludes the proof.
\end{proof}
\begin{remark}[Extension to higher dimensions and more general manifolds]
    Part~\ref{it:2scalei} and~\ref{it:2scaleii} of Theorem~\ref{th:2scale} can readily be generalized to arbitrary dimensions, that is, $\Omega \subset \R^n$ and $(m_\epsilon)_\epsilon \subset L^2(\Omega;\R^m)$ with $n,m \in \N$. Moreover, Part~\ref{it:2scaleiii} also extends to more general target manifolds $\Mcal$ that can be written as the inverse image $\gamma^{-1}(\{0\})$ for some $C^1$-function $\gamma:U \subset \R^m \to \R^{k}$ with $k<m$ that has zero as a regular value. Indeed, we can then use the fundamental theorem of calculus to find that
    \[
    \gamma(m_\epsilon(x+\epsilon \xi))-\gamma(m_\epsilon(x)) = \ (m_\epsilon(x+\epsilon \xi)-m_\epsilon(x)) \cdot \int_0^1 \nabla \gamma(tm_\epsilon(x+\epsilon \xi)+(1-t)m_\epsilon(x))\dd t.
    \]
    By the dominated convergence theorem, the integral converges to $\nabla \gamma(m_0)$ in $L^2(\Omega \times B_R(0);\R^n)$ for any $R>0$ as $\epsilon \to 0$. We can now argue as in Part~\ref{it:2scaleiii} to find that
    \[
    \nabla \gamma(m_0(x))  w(x,z)  = 0 \quad \text{for a.e.~$(x,z) \in \Omega\times Q$,}
    \]
    with $w \in L^2(\Omega;L^2_{\#,0}(Q;\R^m)).$
    Since the rows of $\nabla \gamma$ span the normal space to $\Mcal$, we deduce that $w(x,z) \in T_{m_0(x)}\Mcal$ for a.e.~$z \in Q$.
\end{remark}
\begin{remark}[Two-scale convergence of difference quotients]
    The result of Theorem~\ref{th:2scale} can also be %interpreted as a generalization of 
    compared to the findings in \cite[Section~4]{Vi06}. To see this, we set \(e_i\) the unit vector in direction of the \(x_i\)-axis and define for \(v\in L^2(\Omega;\R^3)\) the difference quotient in direction \(e_i\) as
    \begin{equation*}
        \Delta^i_\veps[v](x)
        \coloneqq \frac{v(x+\veps e_i)-v(x)}{\veps }.
    \end{equation*}
    It then follows from \cite[Lemma~4.1, and Proposition~4.2]{Vi06} that if \(\sequence{v}{\veps}\) is a sequence in \(L^2(\R^3;\R^3)\) converging weakly two-scale to some \(v\in L^2(\R^3\times Q;\R^3)\) and such that 
    \begin{equation*}
        \sup_\veps\norm{\Delta^i_\veps[v]}_{L^2(\R^3)}
        <+\infty,
    \end{equation*}
    then \(\Delta^i_\veps[v]\) converges in the limit \(\veps\to 0\) weakly two-scale to the partial derivative \(\nabla_{x_i} v\). 
    The proof relies on the fact that 
    \begin{equation*}
        \Delta^i_\veps[\psi](x)
        = \frac{\psi(x+\veps e_i,x/\veps)-\psi(x,x/\veps)}{\veps}
    \end{equation*} 
    for test functions which are periodic in the second variable. In other words, \(\Delta^i_\veps[\psi](x)\) is a difference quotient only in the \(x\)-variable. This phenomenon is also reflected in the proof of Theorem~\ref{th:2scale} above in \eqref{eq:VT-2}, where the second integral vanishes for \(\xi\in \Z^3\). 

    Let now \(\xi\in \R^3\) be arbitrary and fixed. Similar considerations as applied in Lemma~\ref{le:Aoperator}--Theorem~\ref{th:2scale} and paired with the proof strategy of \cite[Proposition~4.2]{Vi06}, lead to the following result. We define 
    \begin{equation*}
        \Delta^\xi_\veps[v](x)
        \coloneqq \frac{v(x+\veps \xi)-v(x)}{\veps |\xi|}.
    \end{equation*}
    Then, if \(\sequence{v}{\veps}\in L^2(\R^3;\R^3)\) converges weakly two-scale to \(v\in L^2(\R^3\times Q;\R^3)\) and fulfills
    \begin{equation}\label{eq:ap_est}
        \sup_\veps\norm{\Delta^\xi_\veps[v_\veps]}_{L^2(\R^3)}
        <+\infty,
    \end{equation}
    there exist a $w \in L^2(\R^3;L^2_{\#,0}(Q;\R^3))$ such that
    \begin{equation}\label{eq:vis_aug}
       \Delta^\xi_\veps[v_\epsilon] \wtwoscale (\nabla_{\xi/\abs{\xi}} v(\cdot,z))(x)+\frac{w(x,z+\xi)-w(x,z)}{|\xi|},
    \end{equation}
    where $\nabla_h$ denotes the directional derivative in the direction $h \in \R^3$; note that $w$ may depend on $\xi$ in this case. The difference quotient on the right-hand side lies in the range of the linear operator $S_\xi:L^2(\R^3 ;L^2_{\#}(Q;\R^3)) \to L^2(\R^3;L^2_{\#}(Q;\R^3))$ defined by
    \[
    S_\xi(w)(x,z):= \frac{w(x,z+\xi)-w(x,z)}{|\xi|} \quad \text{for a.e.~$(x,z) \in \R^3\times Q$},
    \]
    which in this case is even bounded.

    We note that condition \eqref{eq:ap_est} is in general stronger than assumption \eqref{eq:bound} in Theorem~\ref{th:2scale}, i.e.,  need not to be satisfied if merely \(\limsup_{\epsilon \to 0} \Ecal_\epsilon(m_\epsilon) < +\infty\).
\end{remark}

We now introduce the pointwise projection map $\pi_{\S^2} : \R^3 \setminus \{ 0 \} \to \S^2$ such that $ \pi_{\S^2}(\theta) = \frac{\theta}{|\theta|}$. Note that away from the origin, $\pi_{\S^2}$ is well-defined and $C^{\infty}$-differentiable. Moreover, for $\theta \in \S^2$ we have that $\nabla \pi_{\S^2}(\theta)=\mathrm{Id} - \theta \otimes \theta$ is the orthogonal projection onto $T_\theta \S^2$. The following result shows that a sequence of the form $(\pi_{\S^2}(m_0(x)+\veps \varphi(x,x/\veps)))_\veps$, with $\varphi$ sufficiently smooth, recovers the two-scale limit where the microscopic oscillations are projected onto the tangent space of the sphere. 
\begin{lemma}[Two-scale convergence of projected sequence]\label{lem:sts_recov}
    Let $m_0 \in H^1(\Omega;\S^2)$ and \(\varphi \in C_c^\infty(\Om, C^\infty_{\#,0}(Q;\R^3))\). Then, denoting 
    \begin{equation*}
        m_\veps^\varphi(x)\coloneqq \pi_{\S^2}(m_0(x)+\veps \varphi(x,x/\veps)),
    \end{equation*}
    the sequence $(\Delta_\epsilon^\rho[m_\veps^\varphi])_\veps$ (see \eqref{eq:vepsform}) strongly $x$-two-scale converges in $L^2(\Omega\times\R^3;\R^3)$ to the limit \(\Delta^\rho[m_0,\nabla \pi_{\S^2}(m_0)\varphi]\) (see \eqref{eq:vform}).
    %, where \hidde{we already defined this in Theorem~\ref{th:2scale}}
    %\begin{align}\label{eq:sts_recov}
    %\begin{split}
    %    \Delta^\rho[m_0,w](x, \xi, z)
    %     = \rho(\xi)^{1/2}\left(\nabla m_0(x) \cdot\frac{\xi}{|\xi|}+\frac{w(x,z+\xi)-w(x,z)}{|\xi|}\right)
    %\end{split}
    %\end{align}
    %is defined for some \(w \in L^2(\Omega;L^2_{\#,0}(Q;\R^3))\) and for a.e.~$(x, \xi, z) \in \Om \times \R^3 \times Q.$
\end{lemma}
\begin{proof} \textit{Step 1: Unconstrained sequence.} 
Let $(\hat{m}^{\varphi}_\veps)_\epsilon \subset H^1(\Omega;\R^3)$ be the oscillating sequence defined by 
\begin{equation}\label{eq:hatmphi}
    \hat{m}^{\varphi}_\veps(x):= m_0(x) + \veps \varphi \left(x, \frac{x}{\veps}\right),
\end{equation}
for $\veps > 0$ and for a.e.\ $x \in \Om$. We will first show that
\[
\Delta^\rho_\epsilon[\hat{m}^{\varphi}_\veps] \xtwoscale \Delta^\rho[m_0,\varphi] \quad \text{in $L^2(\Omega\times\R^3;\R^3)$.}
\]
Indeed, we decompose  $\Delta_\epsilon^\rho[\hat{m}^{\varphi}_\veps] = \Delta_\epsilon^\rho[m_0]+ \veps \Delta_\epsilon^\rho[\varphi(\cdot,\cdot/\epsilon)]$ and prove separately that 
\begin{equation}\label{eq:I1}
    \Delta_\epsilon^\rho[m_0](x,\xi) \, \xtwoscale \, \rho(\xi)^{1/2}\nabla m_0(x) \frac{\xi}{\abs{\xi}} \quad \text{in $L^2(\Omega\times\R^3;\R^3)$}
\end{equation}
and
\begin{equation}\label{eq:I2}
    \veps \Delta_\epsilon^\rho[\varphi(\cdot,\cdot/\epsilon)](x,\xi) \, \xtwoscale \,\rho(\xi)^{1/2}\frac{(\varphi(x,z+\xi)-\varphi(x,z))}{\abs{\xi}} \quad \text{in $L^2(\Omega\times\R^3;\R^3)$}.
\end{equation}
We first focus on \eqref{eq:I1} and observe that
\begin{equation}\label{eq:I-norm}
\begin{split}
    \lim_{\veps \to 0} \|\Delta_\epsilon^\rho[m_0]\|^2_{L^2(\Om \times \R^3)} &= \lim_{\veps \to 0} \int_{\R^3} \int_{\Om_{\veps \xi}} \rho(\xi)\frac{|m_0(x+\epsilon\xi)-m_0(x)|^2}{|\veps\xi|^2}  \dl x \dl \xi \\ &= \lim_{\veps \to 0} \frac{1}{\veps^3} \int_{\Om} \int_{\Om} \rho\left(\frac{y-x}{\veps}\right)\frac{|m_0(x)-m_0(y)|^2}{|x - y|^2} \dl x \dl y \\ & = \int_{\Om} \int_{\R^3} \left| \nabla m_0(x)  \frac{\xi}{|\xi|}\right|^2 \rho(\xi) \dl \xi\dl x,
\end{split}
\end{equation}
where we used the change of variable $y = x + \veps \xi$ for fixed $x \in \Om$, and then we applied \cite[Corollary 1.4]{Pon04b}. Hence, it remains to prove the weak $x$-two-scale convergence in \eqref{eq:I1}. For this, we note that by Theorem~\ref{th:2scale}\,\ref{it:2scaleii} and the uniform bound in \eqref{eq:I-norm}, up to a non-relabeled subsequence, we have that $\Delta^\rho_\epsilon[m_0] \wxtwoscale \Delta^\rho[m_0,w]$ for some $w \in L^2(\Omega;L^2_{\#,0}(Q;\R^3))$. It suffices to prove that $w=0$. This can be done by Jensen's inequality and \eqref{eq:I-norm}, similarly to \eqref{eq:sts_jensen}. Indeed, one has the chain of inequalities
    \begin{align*}
        \int_{\Om} \int_{\R^3} &\left| \nabla m_0(x)  \frac{\xi}{|\xi|}\right|^2 \rho(\xi) \dl \xi \dl x\\
        &\leq \int_{\Omega}\int_{\R^3}\int_{Q} \rho(\xi)\left|\nabla m_0(x) \frac{\xi}{|\xi|} + \frac{w(x,z+\xi)-w(x,z)}{\abs{\xi}}\right|^2\dl z \dl \xi \dl x \\
        &\leq \liminf_{\epsilon \to 0} \norm{\Delta^\rho_\epsilon [m_0]}^2_{L^2(\Omega \times \R^3)}=\int_{\Om} \int_{\R^3} \left| \nabla m_0(x)  \frac{\xi}{|\xi|}\right|^2 \rho(\xi) \dl \xi \dl x,
    \end{align*}
    where %the first inequality is a consequence of Jensen's inequality, 
    the second inequality follows from Lemma~\ref{lem:x_norm_lsc}. %, and the final equality is \eqref{eq:I-norm}.
    Since we now obtain an equality in Jensen's inequality, the function $w(x,z+\xi)-w(x,z)$ must vanish identically, which implies $w$ is zero since it must also have average zero in the second variable. Thus, we have that $\Delta^\rho_\epsilon[m_0] \wxtwoscale \Delta^\rho[m_0,0]$, and together with \eqref{eq:I-norm} this proves \eqref{eq:I1}.

    To show \eqref{eq:I2}, we write
    \begin{align*}
        \veps \Delta_\epsilon^\rho[\varphi(\cdot,\cdot/\epsilon)](x,\xi) &= \mathds{1}_{\Omega_{\epsilon\xi}}(x)\rho(\xi)^{1/2}\frac{(\varphi(x,x/\epsilon+\xi)-\varphi(x,x/\epsilon))}{\abs{\xi}} \\
        &\quad+ \mathds{1}_{\Omega_{\epsilon\xi}}(x)\rho(\xi)^{1/2}\frac{(\varphi(x+\epsilon\xi,x/\epsilon+\xi)-\varphi(x,x/\epsilon+\xi))}{\abs{\xi}}.
    \end{align*}
    The second term is easily estimated using the Lipschitz continuity of $\varphi$ as
    \[
    \abs*{\mathds{1}_{\Omega_{\epsilon\xi}}(x)\rho(\xi)^{1/2}\frac{(\varphi(x+\epsilon\xi,x/\epsilon+\xi)-\varphi(x,x/\epsilon+\xi))}{\abs{\xi}}} \leq C\epsilon \rho(\xi)^{1/2} \quad \text{for a.e.~$(x,\xi) \in \Omega \times \R^3$,}
    \]
    and thus, converges strongly to zero in $L^2(\Omega \times \R^3;\R^3)$. For the first term, we use the strong convergence $\mathds{1}_{\Omega_{\epsilon\xi}}(x)\rho(\xi)^{1/2} \to \rho(\xi)^{1/2}$ in $L^2(\Omega\times\R^3)$ together with Lemma~\ref{le:sxts_nocontinuity} and the Lipschitz continuity of $\varphi$, to find that
    \[
    \mathds{1}_{\Omega_{\epsilon\xi}}(x)\rho(\xi)^{1/2}\frac{(\varphi(x,x/\epsilon+\xi)-\varphi(x,x/\epsilon))}{\abs{\xi}} \, \xtwoscale \,\rho(\xi)^{1/2}\frac{(\varphi(x,z+\xi)-\varphi(x,z))}{\abs{\xi}} \quad \text{in $L^2(\Omega\times\R^3;\R^3)$},
    \]
    which establishes \eqref{eq:I2} and finishes Step 1.
%     We claim that $v$ coincides with the $L^2$-limit as in \eqref{eq:I-norm}.  We now argue as in the proof of Theorem~\ref{th:2scale}\,(ii), where (see from \eqref{eq:smoothpsi})
%     we consider a test function $\psi$ that has compact support in the first argument and that is smooth in the first and third argument, namely 
% \begin{equation}
%         \supp \psi(\cdot,\xi,z) \subset K \Subset \Omega \ \ \text{for a.e.~$(\xi,x) \in \R^3 \times Q$} \quad \text{and} \quad \int_{\R^3}\norm{\psi(\cdot,\xi,\cdot)}_{C^2(\Omega \times Q)}^2\dd \xi <+ \infty.
%     \end{equation}
%     Let $\Omega_{1/R}  = \{ x \in \Omega : \operatorname{dist}(x, \partial \Om) > 1/R\} $, and choose $R>0$ large enough such that $K \times \supp \rho \subset \Omega_{1/R} \times B_R(0)$ and $\epsilon_j <1/R^2$.  
% We deduce that
% \begin{equation}
%     \int_{\R^3} \int_{\R^3} \int_{Q}\left(\rho(\xi)^{1/2}\nabla m_0(x)\frac{\xi}{|\xi|}-v(x,\xi,z)\right) \cdot \psi(x,\xi,z)\dd z\dd x\dd \xi=0,
%     \end{equation}
%     for all $\psi$ satisfying \eqref{eq:smoothpsi} and $S_\rho^*(\psi)=0$ (cf.\ \eqref{eq:Aadjoint}). We now conclude as in \eqref{eq:identification-limit} that ... \ROSS{sentence to be finished because part of \eqref{eq:identification-limit} not so clear for me}.

% By the weakly $x$-two scale convergence and \eqref{eq:I-norm}, we also
% infer the strong $x$-two-scale convergence in $L^2(\Om \times \R^3)$.

% We are left with term $II_\veps$.
\medskip

\textit{Step 2: Constrained sequence.} Due to the definition of $\hat{m}^{\varphi}_\veps$ in \eqref{eq:hatmphi}, we have that $\abs{\hat{m}^{\varphi}_\veps}\geq 1/2$ for $\epsilon >0$ small enough; in the following, we only consider $\epsilon > 0$ such that this latter condition is satisfied. Therefore, if $\tilde{\pi}_{\S^2} \in C^1(\R^3;\R^3)$ is a function such that $\tilde{\pi}_{\S^2}(\theta) = \pi_{\S^2}(\theta)$ for $\abs{\theta} \geq 1/2$, then the projected sequence 
\[
m_\veps^\varphi(x) = \pi_{\S^2}(\hat{m}^{\varphi}_\veps(x))= \tilde{\pi}_{\S^2}(\hat{m}^{\varphi}_\veps(x)) \quad \text{for a.e.~$x \in \Omega$}
\]
is well-defined.
%and $m_\veps^\varphi \to m_0$ strongly in $L^2(\Om; \R^3)$ as $\veps \to 0$, where $m_0 \equiv \pi_{\S^2}(m_0) \in H^1(\Om; \S^2)$. 
Moreover, by the fundamental theorem of calculus we find for each $\xi \in \R^3$ and $x \in \Omega_{\epsilon \xi}$ that
\begin{equation}\label{mvt1}
    m^{\varphi}_\veps(x + \veps \xi) - m^{\varphi}_\veps(x) = \   (\hat{m}^{\varphi}_\veps(x + \veps \xi) - \hat{m}^{\varphi}_\veps(x)) \cdot \int_0^1 \nabla \tilde{\pi}_{\S^2}(t\hat{m}^{\varphi}_\epsilon(x+\epsilon \xi) +(1-t) \hat{m}^{\varphi}_\epsilon(x))\dd t .
\end{equation}
Due to the continuity of translation in $L^2$, we obtain that
\[
\mathds{1}_{\Omega_{\epsilon \xi}}(x)\hat{m}^{\varphi}_\epsilon(x) \to m_0(x) \quad \text{and} \quad \mathds{1}_{\Omega_{\epsilon \xi}}(x)\hat{m}^{\varphi}_\veps(x + \veps \xi) \to m_0(x) \quad \text{in $L^2(\Omega\times B_R(0);\R^3)$,}
\]
for any $R>0$. Since $\nabla \tilde{\pi}_{\S^2}$ is a continuous and bounded function, we now find by the dominated convergence theorem that
\begin{equation}\label{segment-converge}
\mathds{1}_{\Omega_{\epsilon \xi}}(x)\int_0^1 \nabla \tilde{\pi}_{\S^2}(t\hat{m}^{\varphi}_\epsilon(x+\epsilon \xi) +(1-t) \hat{m}^{\varphi}_\epsilon(x))\dd t \to \nabla\tilde{\pi}_{\S^2}(m_0(x)) = \nabla \pi_{\S^2}(m_0(x))
\end{equation}
in $L^2(\Omega\times B_R(0);\R^{3\times 3})$ for any $R>0$. By \eqref{mvt1},  we find that
\begin{equation*}
\Delta_\epsilon^\rho[m_\veps^{\varphi}](x,\xi) =  \ \Delta_\epsilon^\rho[\hat{m}^{\varphi}_\veps](x,\xi) \cdot \int_0^1 \nabla \tilde{\pi}_{\S^2}(t\hat{m}^{\varphi}_\epsilon(x+\epsilon \xi) +(1-t) \hat{m}^{\varphi}_\epsilon(x))\dd t.
\end{equation*}
Then, due to Step 1, the $L^2$-convergence in \eqref{segment-converge}, and the boundedness of $\nabla \tilde{\pi}_{\S^2}$, we obtain from Lemma~\ref{lem:xts_products}\,\ref{xts_products-iii)} that
\begin{equation*}
   \Delta_\epsilon^\rho[m_\veps] \xtwoscale  \nabla \pi_{\S^2}(m_0)\Delta^\rho[m_0,\varphi] \quad\text{in $L^2(\Omega\times\R^3;\R^3)$}. 
\end{equation*}
If we now also use that $\nabla \pi_{\S^2}(m_0(x))  \nabla m_0(x) = \nabla m_0(x)$ for a.e.~$x \in \Omega$, we deduce that the right-hand side is equal to \(\Delta^\rho[m_0,\nabla \pi_{\S^2}(m_0)\varphi]\), as desired.
\end{proof}

\section{\texorpdfstring{$\Gamma$}{Gamma}-convergence of \texorpdfstring{$\Ecal_\epsilon$}{Fcal + Hcal}}\label{sec:gammalim}

With the preparations from the previous section, we are now ready to prove the $\Gamma$-convergence result for the functionals $(\Ecal_\epsilon)_\epsilon$. %To formulate the $\Gamma$-convergence, 
We restrict $\Ecal_\epsilon$ to functions on the sphere, that is, we consider
\begin{equation}\label{energies-Gammaconv}
\Ecal_\epsilon:L^2(\Omega;\S^2) \to \R_{\infty}, \quad \Ecal_\epsilon(m)=\begin{cases}
    \Fcal_\epsilon(m)+\Hcal_\epsilon(m) &\text{if $\Fcal_\epsilon(m)<+\infty$,}\\
    +\infty &\text{else}.
\end{cases}
\end{equation}
Note that $\Hcal_\epsilon(m)$ is well-defined and finite whenever $\Fcal_\epsilon(m)<+\infty$ in light of Lemma~\ref{lem:antisym_by_sym}. We also introduce for shorter notation the matrix-valued tangent bundle $\mathbf{T}\S^2:=\bigcup_{s \in \S^2}\{s\} \times (T_s\S^2)^3$.

For $(s,A)\in \mathbf{T}\S^2$, we define the density
\begin{align}
\begin{split}\label{eq:fhom}
f_{\rm hom}(s,A) := \inf_{v \in L^2_{\#,0}(Q;T_s \S^2)}\int_{\R^3} \int_{ Q}&a(z,z+\xi)\rho(\xi)\frac{|A\xi +v(z+\xi)-v(z)|^2}{|\xi|^2}\\
+&\kappa(z,z+\xi)\frac{(A\xi+v(z+\xi)-v(z))}{|\xi|} \cdot(s \times \nu(\xi)) \dd \xi\dd z.
\end{split}
\end{align}
We then have the following $\Gamma$-convergence result.
\begin{theorem}[$\Gamma$-convergence] \label{th:Gammalimit}
The sequence $(\Ecal_\epsilon)_\epsilon$ $\Gamma$-converges with respect to the $L^2$-topology as $\epsilon \to 0$ to the functional
\[
\Ecal:L^2(\Omega;\S^2) \to \R_{\infty}, \quad \Ecal(m)=\begin{cases}
    \displaystyle \int_{\Omega}f_{\rm hom}(m,\nabla m)\dd x &\text{if $m \in H^1(\Omega;\S^2)$,}\\
    +\infty &\text{else},
\end{cases}
\]
where the homogenized density $f_{\rm hom}$ is given by \eqref{eq:fhom}. 
\end{theorem}

\begin{remark}[Different homogenization and localization scales]\label{rmk:diff_scales}
The proof of the {$\Gamma$-}con\-vergence can be directly extended to the case where the scale of the heterogeneities $\epsilon$ and the nonlocal interaction range $\delta_\epsilon$ satisfy $$\delta_\epsilon/\epsilon \to \lambda \in(0,+\infty) \quad \text{as } \epsilon \to 0.$$ More precisely, with
\begin{align*}
    \Fcal_{\epsilon,\delta_\epsilon}(m)&:=\int_\Omega\int_\Omega a\left(\frac{x}{\epsilon},\frac{y}{\epsilon}\right) \frac{1}{\delta_\epsilon^3}\rho\left(\frac{y-x}{\delta_\epsilon}\right) \frac{|m(y)-m(x)|^2}{|y-x|^2}\dd x\dd y, \\
    \Hcal_{\epsilon,\delta_\epsilon}(m)&:=\int_\Omega\int_\Omega \kappa\left(\frac{x}{\epsilon},\frac{y}{\epsilon}\right)  \frac{1}{\delta_\epsilon^3}\nu\left(\frac{y-x}{\delta_\epsilon}\right)\cdot\frac{(m(y)\times m(x))}{|y-x|}\dd x\dd y,
\end{align*}
it follows that $(\Ecal_{\epsilon,\delta_\epsilon})_\epsilon:=(\Fcal_{\epsilon,\delta_\epsilon}+\Hcal_{\epsilon,\delta_\epsilon})_\epsilon$ $\Gamma$-converges with respect to the $L^2$-topology to
\[
\Ecal^{\lambda}(m):=\int_{\Omega} f^{\lambda}_{\rm hom}(m,\nabla m) \dd x \quad \text{for $m \in H^1(\Omega;\S^2)$},
\]
where
\begin{align*}
f^{\lambda}_{\rm hom}(s,A) := \inf_{v \in L^2_{\#,0}(Q;T_s \S^2)}\int_{\R^3}\int_{Q}&a(z,z+\xi)\rho_{\lambda}(\xi)\frac{|A\xi +v(z+\xi)-v(z)|^2}{|\xi|^2}\\
&\qquad+\kappa(z,z+\xi)\frac{(A\xi+v(z+\xi)-v(z))}{|\xi|} \cdot(s \times \nu_{\lambda}(\xi))\dd z \dd \xi,
\end{align*}
with $\rho_\lambda:=\lambda^{-3}\rho(\cdot/\lambda)$ and $\nu_\lambda=\lambda^{-3}\nu(\cdot/\lambda)$. To see this, we use the substitution $y=x+\frac{\delta_\epsilon}{\lambda}\xi$ to find that
\begin{align*}
\Fcal_{\epsilon,\delta_\epsilon}(m)
%&= \int_{\R^3} \int_{\Omega_{\delta_\epsilon\xi/\lambda}} a\left(\frac{x}{\epsilon},\frac{x}{\epsilon}+\frac{\delta_\epsilon}{\epsilon\lambda}\xi\right) \frac{1}{\lambda^3}\rho\left(\frac{\xi}{\lambda}\right) \frac{|m(x+\delta_\epsilon \xi/\lambda)-m(x)|^2}{|\delta_\epsilon\xi/\lambda|^2}\dd x\dd \xi\\
&=\int_{\R^3} \int_{\Omega} a\left(\frac{x}{\epsilon},\frac{x}{\epsilon}+\frac{\delta_\epsilon}{\epsilon\lambda}\xi\right) \Delta^{\rho_\lambda}_{\delta_\epsilon/\lambda}[m](x,\xi)\dd x\dd \xi
\end{align*}
and
\begin{align*}
\Hcal_{\epsilon,\delta_\epsilon}(m)
&=\int_{\R^3} \int_{\Omega}\left(\kappa\left(\frac{x}{\epsilon},\frac{x}{\epsilon}+\frac{\delta_\epsilon}{\epsilon\lambda}\xi\right)m(x) \times \frac{\nu_\lambda(\xi)}{\rho_\lambda(\xi)^{1/2}}\right)\cdot\Delta^{\rho_\lambda}_{\delta_\epsilon/\lambda}[m](x,\xi) \dd x\dd \xi.
\end{align*}
We recall that the quantity \(\Delta^{\rho_\lambda}_{\delta_\epsilon/\lambda}[m](x,\xi)\) (see \eqref{eq:vepsform}) is given by 
\begin{equation*}
\Delta^{\rho_\lambda}_{\delta_\epsilon/\lambda}[m](x,\xi)
=\1_{\Omega_{\frac{\delta_\epsilon}{\lambda}\xi}}(x)\rho_\lambda(\xi)^{1/2}
\frac{\left(m\left(x+\frac{\delta_\epsilon}{\lambda}\xi\right)-m(x)\right)}{\frac{\delta_\epsilon}{\lambda}|\xi|}.
\end{equation*}
Since $\rho_\lambda$ and $\nu_\lambda$ still satisfy \ref{H2}-\ref{H4} and $\delta_\epsilon/\lambda \to 0$, we can treat the term $\Delta^{\rho_\lambda}_{\delta_\epsilon/\lambda}[m]$ as in the proof of Theorem~\ref{th:Gammalimit} below. The only difference is in the oscillating coefficients, but since $\frac{\delta_\epsilon}{\epsilon\lambda} \to 1$ we still have
\[
a\left(\frac{x}{\epsilon},\frac{x}{\epsilon}+\frac{\delta_\epsilon}{\epsilon\lambda}\xi\right) \xtwoscale a(z,z+\xi) \quad \text{and} \quad \kappa\left(\frac{x}{\epsilon},\frac{x}{\epsilon}+\frac{\delta_\epsilon}{\epsilon\lambda}\xi\right) \xtwoscale \kappa(z,z+\xi).
\]
Thus, the argument follows through in the same manner.
\end{remark}

Before proceeding with the proof of Theorem~\ref{th:Gammalimit}, we present a characterization of the minimizers of \eqref{eq:fhom}, which enables us to exactly identify the homogenized energy density by  solving two separate minimization problems. The argument is inspired by the local case in \cite[Proposition~2.1]{DaD20}. %\hidde{The sign convention of \eqref{eq:vsaformula} seems correct and actually consistent with \cite{DaD20}. Indeed, in their equation (5) they have $\kappa$ positive as well, but later rewrite it as negative using the $\chi$ notation, which is where the confusion came from}
\begin{proposition}[Characterization of minimizers]\label{prop:min_fhom}
    Let $(s,A) \in \mathbf{T}\S^2$, then the unique minimizer of the problem \eqref{eq:fhom} is given by
    \begin{equation}\label{eq:vsaformula}
    v_{s,A}:=Av_a+ s \times v_\kappa,
    \end{equation}
    where $v_a,v_\kappa$ are the unique solutions to, respectively,
    \begin{align*}
    &\inf_{v \in L^2_{\#,0}(Q;\R^3)}\int_{\R^3}\int_{Q}a(z,z+\xi)\rho(\xi)\frac{|\xi +v(z+\xi)-v(z)|^2}{|\xi|^2}\dd z\dd \xi,\\
    &\inf_{v \in L^2_{\#,0}(Q;\R^3)}\int_{\R^3}\int_{Q}a(z,z+\xi)\rho(\xi)\frac{|v(z+\xi)-v(z)|^2}{|\xi|^2}+\kappa(z,z+\xi)\nu(\xi)\cdot \frac{(v(z+\xi)-v(z))}{\abs{\xi}} \dd z\dd \xi.
    \end{align*}
    Moreover, we obtain that
    \[
    f_{\rm hom}(s,A) = (A\overline{T}):A+\sum_{i=1}^{3}s\cdot(\bar{d}_i\times Ae_i)-\int_{\R^3}\int_{Q}a(z,z+\xi)\rho(\xi)\frac{\abs{v_{s,A}(z+\xi)-v_{s,A}(z)}^2}{\abs{\xi}^2}\dd z\dd \xi, 
    \]
    with the matrix $\overline{T} \in \R^{3\times3}$ given by 
    \begin{equation}\label{eq:tensorT}
    \overline{T}:=\int_{\R^3}\int_{Q}a(z,z+\xi)\rho(\xi)\left(\frac{\xi}{\abs{\xi}} \otimes \frac{\xi}{\abs{\xi}}\right)\dd z\dd \xi,
    \end{equation}
    and the vectors $\bar{d}_i \in \R^3$ given by 
    \begin{equation}\label{eq:dmivectors}
    \bar{d}_i:=\int_{\R^3}\int_{Q} \kappa(z,z+\xi)\nu(\xi)\frac{\xi_i}{\abs{\xi}}\dd z\dd \xi \quad \text{for $i\in \{1,2,3\}$.}
    \end{equation}
\end{proposition}

\begin{remark}[Interpretation of $f_{\rm hom}$]
We note that the first two terms of $f_{\rm hom}$, given by
\[
(s,A) \mapsto (A\overline{T}):A +\sum_{i=1}^{3}s\cdot(\bar{d}_i\times Ae_i),
\]
correspond exactly to the energy density obtained in the pure localization limit from \cite[Theorem~2.1]{DDG24} with the specific choice of kernels 
\[
\epsilon^{-3}\bar{\rho}_a\left(\frac{\cdot}{\epsilon}\right) \quad \text{and} \quad \epsilon^{-3}\bar{\nu}_{\kappa}\left(\frac{\cdot}{\epsilon}\right) \quad \text{for $\epsilon >0$,}
\]
where
\[
\bar{\rho}_a(\xi):= \int_{Q}a(z,z+\xi)\rho(\xi)\dd z \quad \text{and} \quad \bar{\nu}_{\kappa}(\xi):=\int_{Q}\kappa(z,z+\xi)\nu(\xi)\dd z.
\]
Indeed, the measure $\mu \in \Mcal(\S^2)$ in \cite[Theorem~2.1]{DDG24} in this specific case can be explicitly computed to be
\[
\mu(B):=\int_{B} \int_0^\infty \bar{\rho}_a(t\sigma)t^{2}\dd t \dd \Hcal^2(\sigma) \quad \text{for a Borel set $B \subset \S^2$,}
\]
cf.~\cite[Example~2]{Pon04b}, which yields
\[
\int_{\S^2} \abs{A \sigma}^2\dd \mu(\sigma) =  \int_{\R^3}\int_{Q}a(z,z+\xi)\rho(\xi)\frac{\abs{A\xi}^2}{\abs{\xi}^2}\dd z\dd \xi = (A\overline{T}):A.
\]
Intuitively, this means we average out the contributions of the periodic coefficients $a$ and $\kappa$, and then simply take the pointwise limit of the energies. On the other hand, the third term of $f_{\rm hom}$ given by
\[
(s,A) \mapsto -\int_{\R^3}\int_{Q}a(z,z+\xi)\rho(\xi)\frac{\abs{v_{s,A}(z+\xi)-v_{s,A}(z)}^2}{\abs{\xi}^2}\dd z\dd \xi
\]
is non-positive and accounts for the homogenization effect. Indeed, it represents how much energy we can save by minimizing over each periodic cell.
\end{remark}

\begin{proof}[Proof of Proposition \ref{prop:min_fhom}]
    Consider the functional $\Fcal_{s,A}:L^2_{\#,0}(Q;T_s\S^2) \to \R_{\infty}$ given by
    \begin{align*}
    \begin{split}
    \Fcal_{s,A}(v):=\int_{\R^3}\int_{Q}&a(z,z+\xi)\rho(\xi)\frac{|A\xi +v(z+\xi)-v(z)|^2}{|\xi|^2}\\
    +&\kappa(z,z+\xi)\frac{(A\xi+v(z+\xi)-v(z))}{|\xi|} \cdot(s \times \nu(\xi)) \dd z\dd \xi.
    \end{split}
    \end{align*}
    We first prove that it has a unique minimizer. For the coercivity, we use Young's inequality to obtain that
    \begin{align*}
        &\int_{\R^3} \int_{Q} \abs*{\kappa(z,z+\xi)\nu(\xi) \cdot \left(\frac{(A\xi+v(z+\xi)-v(z))}{\abs{\xi}}\times s\right)} \dd z\dd \xi \\ = & \int_{\R^3} \int_{Q} \abs*{ a^{1/2}(z, z + \xi) \rho(\xi)^{1/2} \frac{(A\xi+v(z+\xi)-v(z))}{\abs{\xi}}  \cdot \left( \frac{\kappa(z,z+\xi)\nu(\xi)}{a^{1/2}(z, z + \xi) \rho(\xi)^{1/2}} \times s\right) } \dd z\dd \xi \\ \leq & \, \frac{1}{2} \int_{\R^3}\int_{Q}a(z,z+\xi)\rho(\xi)\frac{|A\xi +v(z+\xi)-v(z)|^2}{|\xi|^2} \dd z\dd \xi  +  \frac{1}{2 a_0} \norm*{\kappa}^2_{L^\infty(\R^3\times \R^3)}\norm*{\frac{\nu}{\rho^{1/2}}}^{2}_{L^{2}(\R^3)},
\end{align*}
where we remind that \(a\ge a_0>0\), and where we used Assumption \ref{H4}.

% \begin{align*}
%  \int_{\R^3}\int_\Om\frac{|\kappa(z,z+\xi)\nu(\xi) \times s|^2}{a(z, z + \xi)\rho(\xi)}  \dd x \dd \xi
%     \le \norm{\nu}_{L^1(\R^3)}\norm*{\frac{\nu}{\rho}}_{L^{\infty}(\R^3)}\norm{m}^2_{L^2(\Om)}.
% \end{align*}
    
    % and then 
    % \begin{align*}
    %     2\abs*{\kappa(z,z+\xi)\nu(\xi) \cdot \left(\frac{(A\xi+v(z+\xi)-v(z))}{\abs{\xi}}\times s\right)} &\leq a(z,z+\xi)\rho(\xi)\frac{|A\xi +v(z+\xi)-v(z)|^2}{|\xi|^2} \\
    %     &\quad+\frac{\abs{\nu(\xi)}^2\norm{\kappa}^2_{L^\infty(Q\times Q)}}{a_0\rho(\xi)}, 
    % \end{align*}
    % where we multiplied and divided by the quantity $a(z,z+\xi)\rho(\xi)$ and then used \ref{H1}.
    Hence, %by \ref{H3}-\ref{H4}, there is a $C>0$ such that
    \begin{equation}\label{eq:fsalowerbound}
    \Fcal_{s,A}(v) \geq \frac{a_0}{2} \int_{\R^3}\int_{Q}\rho(\xi)\frac{|A\xi +v(z+\xi)-v(z)|^2}{|\xi|^2}\dd z\dd \xi - C, 
    \end{equation}
    with 
    \begin{equation*}
        C:= \frac{1}{2 a_0} \norm*{\kappa}^2_{L^\infty(\R^3\times \R^3)}\norm*{\frac{\nu}{\rho^{1/2}}}^{2}_{L^{2}(\R^3)}.
    \end{equation*}
    Due to the lower control on $\rho$ from \ref{H2}, for the right-hand side of  \eqref{eq:fsalowerbound} we rely on a nonlocal Poincar\'{e}-Wirtinger-type inequality, in a similar way to inequality \eqref{eq:nonlocal_poincare} in Appendix~\ref{app:a}, to infer that $\Fcal_{s,A}$ is weakly coercive in $L^2_{\#,0}(Q;T_s\S^2)$. Together with the weak lower semicontinuity of $\Fcal_{s,A}$ (see~e.g.~\cite[Theorem~6.54]{FoL07}), we find that $\Fcal_{s,A}$ admits a minimizer by the direct method in the calculus of variations. Moreover, due to the strict convexity of $\Fcal_{s,A}$, we have that the minimizer is unique and we denote it by $v_{s,A}$. 
    % An application of Jensen's inequality to \eqref{eq:fsalowerbound} yields
    % \[
    % f_{\rm hom}(s,A)=\Fcal_{s,A}(v_{s,A})\geq \frac{a_0}{2} \int_{\R^3}\rho(\xi)\frac{|A\xi|^2}{|\xi|^2}\dd \xi - C\geq c|A|^2-C,
    % \]
    % for some $c>0$ depending on $\rho$ and $a_0$.
    
    In order to obtain \eqref{eq:vsaformula}, we observe that the convexity of $\Fcal_{s,A}$ implies that $v_{s,A}$ is characterized by the weak Euler-Lagrange equation 
    \begin{align}
    \begin{split}\label{eq:el}
        \int_{\R^3}\int_{Q}2 \,a(z,z+\xi)\rho(\xi)&\frac{A\xi+v_{s,A}(z+\xi)-v_{s,A}(z)}{\abs{\xi}}\cdot \frac{(\varphi(z+\xi)-\varphi(z))}{\abs{\xi}}\\
        +&\kappa(z,z+\xi)\nu(\xi)\cdot \left(\frac{(\varphi(z+\xi)-\varphi(z))}{\abs{\xi}} \times s\right)\dd z\dd \xi=0
        \end{split}
    \end{align}
    for all $\varphi \in C_{\#}^{\infty}(Q;T_s\S^2)$. On the other hand, it can be verified that $Av_a+ s \times v_\kappa$ lies in $L^2_{\#,0}(Q;T_s\S^2)$ given that the columns of $A$ lie in the tangent space $T_s\S^2$ and $s$ is orthogonal to $T_s\S^2$. Additionally, writing down the weak Euler-Lagrange equations for $v_a$ and $v_\kappa$, respectively, one can directly verify that $Av_a+s \times v_\kappa$ satisfies \eqref{eq:el}. Hence, we conclude that \eqref{eq:vsaformula} indeed holds. Finally, to obtain the formula for $f_{\rm hom}$, we can compute
    \begin{align*}
        f_{\rm hom}(s,A)= \int_{\R^3}\int_{Q}&a(z,z+\xi)\rho(\xi)\frac{|A\xi +v_{s,A}(z+\xi)-v_{s,A}(z)|^2}{|\xi|^2}\\
+&\kappa(z,z+\xi)\frac{(A\xi+v_{s,A}(z+\xi)-v_{s,A}(z))}{|\xi|} \cdot(s \times \nu(\xi)) \dd z\dd \xi\\
= \int_{\R^3}\int_{Q}&a(z,z+\xi)\rho(\xi)\frac{|A\xi|^2 +2A\xi \cdot (v_{s,A}(z+\xi)-v_{s,A}(z))+|v_{s,A}(z+\xi)-v_{s,A}(z)|^2}{|\xi|^2}\\
+&\kappa(z,z+\xi)\frac{(A\xi+v_{s,A}(z+\xi)-v_{s,A}(z))}{|\xi|} \cdot(s \times \nu(\xi)) \dd z\dd \xi\\
= \int_{\R^3}\int_{Q}&a(z,z+\xi)\rho(\xi)\frac{|A\xi|^2 -|v_{s,A}(z+\xi)-v_{s,A}(z)|^2}{|\xi|^2}\\
+&\kappa(z,z+\xi)\frac{A\xi}{|\xi|} \cdot(s \times \nu(\xi))\dd z\dd \xi,
    \end{align*}
    where we have used \eqref{eq:el} with $\varphi = v_{s,A}$ in the third equality.  Using that $\abs{A\xi}^2=(A(\xi\otimes\xi)):A$, we find exactly that 
    \[
    \int_{\R^3}\int_{Q}a(z,z+\xi)\rho(\xi)\frac{\abs{A\xi}^2}{\abs{\xi}^2}\dd z\dd \xi = (A\overline{T}):A
    \]
    with $\overline{T}$ as in \eqref{eq:tensorT}, while a direct computation shows that
    \begin{align*}
        \int_{\R^3}\int_{Q}\kappa(z,z+\xi)\frac{A\xi}{|\xi|} \cdot(s \times \nu(\xi))\dd z\dd \xi &= \int_{\R^3}\int_{Q} \kappa(z,z+\xi) \,s\cdot\left(\nu(\xi) \times \frac{A\xi}{\abs{\xi}}\right) \dd z\dd \xi\\
        &=\sum_{i=1}^{3} \int_{\R^3}\int_{Q} \kappa(z,z+\xi)\, s\cdot\left(\nu(\xi)\frac{\xi_i}{\abs{\xi}} \times Ae_i\right) \dd z\dd \xi\\
        &=\sum_{i=1}^{3}s\cdot(\bar{d}_i\times Ae_i),
    \end{align*}
    with the vectors $\bar{d}_i$ as in \eqref{eq:dmivectors}.
\end{proof}

The preceding Proposition~\ref{prop:min_fhom} also implies an alternative formulation of the \(\Gamma\)-limit in Theorem~\ref{th:Gammalimit}, where the infimum is taken over functions in \(L^2(\Om;L^2_{\#,0}(Q,T_{m}\S^2))\). This reformulation will be particularly useful in the construction of the recovery sequence in Theorem~\ref{th:Gammalimit}, since it allows to circumvent the application of an abstract selection criterion. Before stating the result, we introduce the formal notations
\begin{align}
    \mathcal{F}(m,w)
    &\coloneqq \int_{\Om}\int_{\R^3}\int_Q a(z,z+\xi)\rho(\xi)\frac{\abs{\nabla m(x)\xi+ w(x,z+\xi)-w(x,z)}^2}{\abs{\xi}^2}
    \dl z\dl \xi\dl x,   \label{F_2} \\
    \mathcal{H}(m,w)
    &\coloneqq \int_{\Om}\int_{\R^3}\int_Q \kappa(z,z+\xi)\frac{(\nabla m(x)\xi+w(x,z+\xi)-w(x,z))}{|\xi|}\cdot(m(x)\times \nu(\xi)) \dd z \dl \xi\dl x. \label{H_2}
\end{align}

\begin{lemma}\label{lem:Glim_refor}
    With the definition of \(f_{\rm hom}\) in \eqref{eq:fhom} there holds for all \(m \in H^1(\Omega;\S^2)\) the equality
    \begin{align}\label{eq:Glim_refor}
    \begin{split}
        \Ecal(m)=\int_{\Omega}f_{\rm hom}(m,\nabla m)\dd x
        = \min_{w\in  L^2(\Om;L^2_{\#,0}(Q,T_{m}\S^2))}
        \left(\mathcal{F}(m,w) + \mathcal{H}(m,w)\right).
    \end{split}
    \end{align}
\end{lemma}
\begin{proof}
   The fact that the left-hand side of \eqref{eq:Glim_refor} is smaller than the right-hand side follows straightforwardly from the definition of \(f_{\rm hom}\) in \eqref{eq:fhom}, since every \(w\in L^2(\Om;L^2_{\#,0}(Q,T_{m}\S^2))\) fulfills \(w(x,\cdot)\in L^2_{\#,0}(Q;T_{m(x)} \S^2)\) for a.e.~\(x\in \Om\). 

   The reverse inequality is a direct consequence of Proposition~\ref{prop:min_fhom}. Indeed, the function $$ w(x,z)\coloneqq \nabla m(x) v_a(z) +m(x)\times v_\kappa(z) $$ minimizes the right-hand side of \eqref{eq:fhom} for every \((s,A)=(m(x),\nabla m(x))\), and simultaneously belongs to \(L^2(\Om;L^2_{\#,0}(Q;T_m\S^2))\), since \(m\in H^1(\Om;\S^2)\).
\end{proof}
We are now ready to prove the $\Gamma$-convergence result.
\begin{proof}[Proof~of~Theorem~\ref{th:Gammalimit}]
%\textit{Part (i).} The equi-coercivity is a direct consequence of Lemma~\ref{le:compactness}.
We split the proof into two parts. \smallskip

\textit{Liminf-inequality}. Let $(m_\epsilon)_\epsilon \subset L^2(\Omega;\S^2)$ be a sequence converging to $m_0 \in L^2(\Omega;\S^2)$ and suppose without loss of generality that
\[
\liminf_{\epsilon \to 0} \Ecal_\epsilon(m_\epsilon) < +\infty.
\]
By Theorem~\ref{th:2scale}, we find that $m_0 \in H^1(\Omega;\S^2)$ and there exists a $w \in L^2(\Omega;L^2_{\#,0}(Q;T_{m_0}\S^2))$ such that, up to a non-relabeled subsequence, $\Delta^\rho_\epsilon [m_\epsilon] \wxtwoscale \Delta^\rho[m_0,w]$ in $L^2(\Omega\times\R^3;\R^3)$, cf.~\eqref{eq:vform}. Moreover, taking $a_\epsilon(x,\xi):=a(x/\epsilon,x/\epsilon+\xi)$, by Lemma~\ref{le:sxts_nocontinuity} we find that $a_\epsilon^{1/2}(x,\xi) \xtwoscale a(z,z+\xi)^{1/2}$ in $L^2(\Omega\times B_R(0))$ for any $R>0$. 
From Lemma~\ref{lem:xts_products}\,\ref{xts_products-ii)}, we now find that 
%$(a_\epsilon^{1/2}\Delta^\rho_\epsilon [m_\epsilon])_\epsilon$ weakly $x$-two-scale converges in $L^2(\Omega\times\R^3;\R^3)$ to
\[
a_\epsilon^{1/2}(x,\xi)\Delta^\rho_\epsilon [m_\epsilon](x,\xi) \wxtwoscale a(z,z+\xi)^{1/2}\rho(\xi)^{1/2}\left(\nabla m_0(x)\frac{\xi}{|\xi|}+\frac{w(x,z+\xi)-w(x,z)}{|\xi|}\right)
\]
in $L^2(\Omega\times\R^3;\R^3)$. Using Lemma~\ref{lem:x_norm_lsc}, we arrive at
\begin{align}
\begin{split}\label{eq:Fliminf}
\liminf_{\epsilon \to 0} \Fcal_\epsilon(m_\epsilon) &= \liminf_{\epsilon \to 0} \int_{\Omega}\int_{\R^3} \abs{a_\epsilon(x,\xi)^{1/2}\Delta^\rho_\epsilon [m_\epsilon](x,\xi)}^2\dd \xi\dd x \\
&\geq \int_{\Omega}\int_{\R^3}\int_{Q} a(z,z+\xi)\rho(\xi)\frac{\abs{\nabla m_0(x)\xi+w(x,z+\xi)-w(x,z)}^2}{|\xi|^2} \dd z \dd \xi \dd x\\
&= \Fcal(m_0,w).
\end{split}
\end{align}

For the antisymmetric term, we write
\begin{align*}
    &\1_{\Omega_{\epsilon\xi}}(x)\kappa\left(\frac{x}{\epsilon},\frac{x}{\epsilon}+\xi\right)\nu(\xi)\cdot\frac{(m_\epsilon(x+\epsilon\xi)\times m_\epsilon(x))}{\epsilon|\xi|}\\
   &\quad=\Delta^\rho_\epsilon[m_\epsilon](x,\xi)\cdot \left(\kappa\left(\frac{x}{\epsilon},\frac{x}{\epsilon}+\xi\right)m_\epsilon(x) \times \frac{\nu(\xi)}{\rho(\xi)^{1/2}}\right).
\end{align*}
In light of \ref{H4}, Lemma~\ref{le:sxts_nocontinuity} and Lemma~\ref{lem:xts_products}\,\ref{xts_products-iii)}, we find that 
\[
\kappa\left(\frac{x}{\epsilon},\frac{x}{\epsilon}+\xi\right)m_\epsilon(x) \times \frac{\nu(\xi)}{\rho(\xi)^{1/2}}\xtwoscale
\kappa(z,z+\xi) \,m_0(x) \times \frac{\nu(\xi)}{\rho(\xi)^{1/2}} \quad \text{in $L^2(\Omega\times\R^3;\R^3)$}.
\]
 Together with the convergence $\Delta^\rho_\epsilon [m_\epsilon] \wxtwoscale \Delta^\rho[m_0,w]$, we deduce from Lemma~\ref{lem:xts_products}\,\ref{xts_products-i)} that
\begin{align}
\begin{split}\label{eq:Hlim}
    &\lim_{\epsilon \to 0}\Hcal_\epsilon(m_\epsilon) \\
    &\quad= \lim_{\epsilon \to 0} \int_{\Omega}\int_{\R^3} \1_{\Omega_{\epsilon\xi}}(x)\kappa\left(\frac{x}{\epsilon},\frac{x}{\epsilon}+\xi\right)\nu(\xi)\cdot\frac{(m_\epsilon(x+\epsilon\xi)\times m_\epsilon(x))}{\epsilon|\xi|}\dd \xi\dd x \\
    &\quad=\int_{\Omega}\int_{\R^3}\int_{Q} \kappa(z,z+\xi)\frac{\nabla m_0(x)\xi+w(x,z+\xi)-w(x,z)}{|\xi|}\cdot(m_0(x)\times \nu(\xi)) \dd z \dd \xi \dd x\\
    &\quad= \Hcal(m_0,w).
\end{split}
\end{align}
Adding \eqref{eq:Fliminf} and \eqref{eq:Hlim} together and using Lemma~\ref{lem:Glim_refor}, we deduce that
\[
\liminf_{\epsilon \to 0} \Ecal_\epsilon(m_\epsilon) 
\geq 
    \mathcal{F}(m_0,w) + \mathcal{H}(m_0,w)
\geq\Ecal(m_0).
\]
\smallskip

\textit{Recovery sequence}. If $m_0 \in L^2(\Omega;\S^2) \setminus H^1(\Omega;\S^2)$, then $\mathcal{E}(m_0)= \infty$ and there is nothing to prove. 
Then, let $m_0 \in H^1(\Omega;\S^2)$ and assume that \(\mathcal{E}(m_0)<+ \infty\). 
We rely on Lemma~\ref{lem:sts_recov}. In fact, for any \(\varphi \in C_c^\infty(\Om; C^\infty_{\#,0}(Q;\R^3))\), considering the sequence $m_\veps^\varphi = \pi_{\S^2}(m_0+\veps \varphi(x,x/\veps))$ for every $\veps > 0$, we conclude from Lemma~\ref{lem:sts_recov} and Lemma~\ref{lem:xts_products}\,\ref{xts_products-iii)} that 
\begin{align}
    \begin{split}\label{eq:Flim_recov}
    &\lim_{\epsilon \to 0} \Fcal_\epsilon(m_\epsilon^\varphi) \\
    & = \lim_{\epsilon \to 0} \int_{\R^3} \int_\Om\abs{a_\epsilon(x,\xi)^{1/2}\Delta^\rho_\epsilon [m_\epsilon^\varphi](x,\xi)}^2\dd x\dd \xi\\
    &  =  \int_{\R^3}\int_\Om\int_Q a(z,z+\xi)\rho(\xi)\frac{\abs{\nabla m_0(x)\xi+ \nabla \pi_{\S^2}(m_0(x))\left(\varphi(x,z+\xi)-\varphi(x,z)\right)}^2}{\abs{\xi}^2}
    \dl z\dl x\dl \xi\\
    & = \mathcal{F}(m_0,\nabla \pi_{\S^2}(m_0)\varphi).
    \end{split}
\end{align}
Moreover, similar to \eqref{eq:Hlim} we also have that
\begin{align}
    &\lim_{\epsilon \to 0}\Hcal_\epsilon(m_\epsilon^\varphi)\nonumber \\
    &= \lim_{\epsilon \to 0} \int_{\R^3}\int_{\Om} \1_{\Omega_{\epsilon\xi}}(x)\kappa\left(\frac{x}{\epsilon},\frac{x}{\epsilon}+\xi\right)\nu(\xi)\cdot\frac{(m_\epsilon^\varphi(x+\epsilon\xi)\times m_\epsilon^\varphi(x))}{\epsilon|\xi|}\dd x\dd \xi\label{eq:Hlim_recov} \\
    &=\int_{\R^3}\int_\Om\int_Q \kappa(z,z+\xi)\frac{\nabla m_0(x)\xi+\nabla \pi_{\S^2}(m_0(x))\left(\varphi(x,z+\xi)-\varphi(x,z)\right)}{|\xi|}\cdot(m_0(x)\times \nu(\xi)) \dd z \dd x \dd \xi\nonumber\\
    &= \mathcal{H}(m_0,\nabla \pi_{\S^2}(m_0)\varphi).\nonumber
\end{align}
The claim now essentially follows from Lemma~\ref{lem:Glim_refor} and by exploiting that, according to Proposition~\ref{prop:density}, the space \(C_c^\infty(\Om; C^\infty_{\#,0}(Q;\R^3))\) is dense in \(H^\rho_{\#,0}\) with respect to the norm \(\norm{\cdot}_\rho\) defined in \eqref{eq:seminorm}. We make this precise in the following. Taking \(w\in L^2(\Om;L^2_{\#,0}(Q,T_{m}\S^2))\) that minimizes the right-hand side of \eqref{eq:Glim_refor}, we first demonstrate that \(w\in H^\rho_{\#,0}\). Indeed, because \(\inner{w}_{Q}=0\), there holds 
\begin{align*}
    \norm{w}_\rho^2
    &\le 
    C\int_{\R^3}\int_\Om\int_Qa(z,z+\xi)\rho(\xi)\frac{|\nabla m_0(x)\xi +w(x,z+\xi)-w(x,z)|^2}{|\xi|^2}{\dd z \dd x \dd \xi}\\
    &\quad+ C\int_{\R^3}\int_\Om\int_Qa(z,z+\xi)\rho(\xi)|\nabla m_0(x)|^2{\dd z \dd x \dd \xi}\\
    &\le C \mathcal{F}(m_0,w) + C\norm{\nabla m_0}^2_{L^2(\Om)},
\end{align*}
with $\mathcal{F}(m_0,w)$ as defined in \eqref{F_2}. 
Additionally, with similar computations as in Lemma~\ref{lem:antisym_by_sym}, we have
\begin{align*}
    \mathcal{H}(m_0,w)
    \ge -\frac{1}{2} \mathcal{F}(m_0,w) - C\norm{m_0}^2_{L^2(\Om)},
\end{align*}
with $ \mathcal{H}(m_0,w)$ as in \eqref{H_2}, from which we conclude, adding \(\mathcal{F}(m_0,w)\) to both sides, that
\begin{align*}
    \mathcal{F}(m_0,w)
    \le 2\mathcal{E}(m_0) + C\norm{m_0}^2_{L^2(\Om)}.
\end{align*}
Overall, from the assumption that \(\mathcal{E}(m_0)<+ \infty\), it follows 
\begin{equation*}
    \norm{w}_\rho
    \le C\mathcal{E}(m_0) + C\norm{m_0}^2_{H^1(\Om)}
    <+ \infty.
\end{equation*}
Hence, by the density result in Proposition~\ref{prop:density}, there exist functions \(\seq{\varphi}\subset C_c^\infty(\Om; C^\infty_{\#,0}(Q;\R^3))\) such that \(\norm{w-\varphi_k}_\rho\to 0\) as $k \to \infty$. 
Using that $w(x,\cdot) \in T_{m_0(x)}\S^2$ and that \(\nabla \pi_{\S^2}(m_0(x))\) is the orthogonal projection onto \(T_{m_0(x)}{\S^2}\) for a.e.~$x \in \Omega$, we infer that also
\[
 \norm{w-\nabla\pi_{\S^2}(m_0)\varphi_k}_\rho\to 0 \quad \text{as $k \to \infty$}.
\]
% In view of the symmetric contribution, we note that the vanishing mean \(\inner{w}_{Q}=0\) implies that
% \begin{align*}
%     \mathcal{F}(m_0,w)
%     = \int_{\R^3}\int_\Om\int_Q a(z,z+\xi)\rho(\xi)\frac{\abs{\nabla m_0(x)\xi}^2+\abs{w(x,z+\xi)-w(x,z)}^2}{\abs{\xi}^2}
%     \dl z\dl x\dl \xi.
% \end{align*}
% Additionally using that \(\nabla \pi_{\S^2}(\theta)\) is the projection onto \(T_\theta{\S^2}\) and the fact that $a(x,y) \geq a_0 >0$ for a.e.~$x,y \in \R^3$ by hypothesis \ref{H1}, in view of \(\tilde{\mathcal{F}}(m_0,\varphi_k)\) we similarly find
% \begin{equation*}
%     \tilde{\mathcal{F}}(m_0,\varphi_k)
%     \le \int_{\R^3}\int_\Om\int_Q a(z,z+\xi)\rho(\xi)\frac{\abs{\nabla m_0(x)\xi}^2+\abs{\varphi_k(x,z+\xi)-\varphi_k(x,z)}^2}{\abs{\xi}^2}
%     \dl z\dl x\dl \xi.
% \end{equation*}
Then, because \(a \in L^{\infty}(\R^3 \times \R^3)\) by hypothesis \ref{H1}, we find with the reverse triangle inequality that
 \begin{align}\label{eq:symm_diff_van}
    |\mathcal{F}(m_0,w)&^{1/2}-\mathcal{F}(m_0,\nabla\pi_{\S^2}(m_0)\varphi_k)^{1/2}|
    \le \norm{a}_{L^{\infty}(\R^3 \times \R^3)}
    \norm{w-\nabla\pi_{\S^2}(m_0)\varphi_k}_\rho
    \to 0.
 \end{align}
 % \begin{align}\label{eq:symm_diff_van}
 %    |\mathcal{F}(m_0,w)&^{1/2}-\mathcal{F}(m_0,\varphi_k)^{1/2}|
 %    \le \norm{a}_{L^{\infty}(\R^3 \times \R^3)}
 %    \Big|\norm{w}_\rho^2-\norm{\varphi_k}_\rho^2\Big|
 %    \to 0.
 % \end{align}
 % To estimate the difference of the antisymmetric contributions, we observe that first applying H{\"o}lder's inequality on \(\Om\times Q\) and then Jensen's inequality (using that \(\rho(\xi)\dl \xi\) defines a probability measure on \(\R^3\) by hypothesis \ref{H2}) we have
 % \begin{align*}
 %     \begin{split}
 %     &\Bigg(\int_{\R^3}\int_\Om\int_Q 
 %     \rho(\xi)\frac{|(w(x,z+\xi)-w(x,z))-\nabla \pi_{\S^2}(m_0(x))\left(\varphi_k(x,z+\xi)-\varphi_k(x,z)\right)|}{|\xi|}   \dd z \dd x \dd \xi\Bigg)^2\\
 %     &\le C \int_{\R^3}\int_\Om\int_Q 
 %      \rho(\xi)\frac{|(w(x,z+\xi)-w(x,z))-\nabla \pi_{\S^2}(m_0(x))\left(\varphi_k(x,z+\xi)-\varphi_k(x,z)\right)|^2}{|\xi|^2}
 %      \dd z \dd x\dd \xi.
     % & =\Bigg(\int_{\R^3}\int_\Om\int_Q 
     % \rho(\xi)\frac{|\nabla \pi_{\S^2}(m_0(x))[(w(x,z+\xi)-\varphi_k(x,z+\xi))-(w(x,z)-\varphi_k(x,z))]|}{|\xi|}   \dd z \dd x \dd \xi\Bigg)^2\\
     % &\le C \int_{\R^3}\int_\Om\int_Q 
     % \rho(\xi)\frac{|(w(x,z+\xi)-\varphi_k(x,z+\xi))-(w(x,z)-\varphi_k(x,z))|^2}{|\xi|^2}
     % \dd z \dd x\dd \xi.
 %     \end{split}
 % \end{align*}
 % With this, employing also that \(\kappa \in L^{\infty}(\R^3 \times \R^3)\) and \(\nu/\rho\in L^\infty(\R^3)\) by hypothesis \ref{H1} and \ref{H4}, we infer that
 To estimate the difference of the antisymmetric contributions, we observe that \begin{align}
     &|\mathcal{H}(m_0,w)-\mathcal{H}(m_0,\nabla \pi_{\S^2}(m_0)\varphi_k)|\nonumber \\ 
     & \leq \norm{\kappa}_{L^{\infty}(\R^3 \times \R^3)}\int_{\R^3}\int_\Om\int_Q  |\nu(\xi)|\nonumber  \\
     &\qquad\qquad\qquad\qquad\qquad\frac{|(w(x,z+\xi)-w(x,z))-\nabla \pi_{\S^2}(m_0(x))\left(\varphi_k(x,z+\xi)-\varphi_k(x,z)\right)|}{|\xi|} \dd z \dd x\dd \xi \nonumber \\ 
     & \leq \abs{\Omega} \norm{\kappa}_{L^{\infty}(\R^3 \times \R^3)} \norm*{\frac{\nu}{\rho^{1/2}}}_{L^{2}(\R^3)}^2 \norm{w-\nabla \pi_{\S^2}(m_0)\varphi_k}_\rho^2 \label{eq:antisymm_diff_van} \\ 
     & \leq C \norm{w-\nabla \pi_{\S^2}(m_0)\varphi_k}_\rho^2 \to 0,\nonumber
 \end{align}
 where in the first inequality we used that $|m_0| = 1$ a.e.\ in $\Om$, while in the second inequality we multiplied and divided by the quantity $\rho(\xi)^{1/2}$, and then used Cauchy-Schwarz inequality. Moreover, in the last inequality we exploited that $\kappa \in L^{\infty}_{\#}(Q\times Q) $ and Assumption \ref{H4}.
 
 We now conclude from \eqref{eq:Flim_recov}--%, \eqref{eq:Hlim_recov}, \eqref{eq:symm_diff_van}, 
 \eqref{eq:antisymm_diff_van} and Lemma~\ref{lem:Glim_refor} that for all \(\delta>0\) there exists a \(\varphi\in C_c^\infty(\Om; C^\infty_{\#,0}(Q;\R^3))\) such that the approximate \(\limsup\)-inequality
 \begin{equation*}
     \limsup_{\epsilon \to 0} \Ecal_\epsilon(m_\epsilon^\varphi) 
    = \mathcal{F}(m_0,\nabla \pi_{\S^2}(m_0)\varphi) + \mathcal{H}(m_0,\nabla \pi_{\S^2}(m_0)\varphi)
    \le\Ecal(m_0)+\delta
 \end{equation*}
 holds.
 By standard properties of \(\Gamma\)-convergence (cf. \cite[Section~1.2]{Br02}), this concludes the proof.
 \end{proof}

\section*{Acknowledgements}
\noindent 
This research was funded in whole
by the Austrian Science Fund (FWF) projects \href{https://doi.org/10.55776/F65}{10.55776/F65}, \href{https://doi.org/10.55776/Y1292}{10.55776/Y1292}, and \href{https://doi.org/10.55776/P34609}{10.55776/P34609}. For open-access purposes, the authors have applied a CC BY public copyright license to any author-accepted manuscript version arising from this
submission. \\
%The research of R.G. and L.H. was funded by the Austrian Science Fund (FWF) through project\\ \href{https://doi.org/10.55776/P34609}{10.55776/P34609}. 
L.H. also acknowledges support from Grant PCI2024-155023-2 funded by the European Union and MCIN/AEI/10.13039/501100011033. Most of the research was done while H.S. was affiliated with TU Wien, and H.S. also acknowledges support from the Belgian Research Fund F.R.S.-FNRS through the ZESTE project during the finalizing stages of the paper. R.G. and L.H. thank TU Wien and MedUni Wien for their support and hospitality. 
%H.S. was funded by the Austrian Science Fund (FWF) projects \href{https://doi.org/10.55776/F65}{10.55776/F65} and \href{https://doi.org/10.55776/Y1292}{10.55776/Y1292}.

%%%%%%%%%%%%%%%%%%%%%%%%%%%%%%%%%%%%%%%%

\appendix
\section{Auxiliary computations}\label{app:a}
We recall the notation
\begin{align*}
    \Fcal_\epsilon(m)&=\int_{\R^3 } \int_{\Omega_{\veps \xi}} a\left(\frac{x}{\epsilon},\frac{x}{\epsilon} + \xi \right) \rho(\xi)  \frac{|m(x + \veps \xi)-m(x)|^2}{|\veps \xi|^2}\dd x\dd \xi, \\
    \Hcal_\epsilon(m)&=\int_{\R^3 } \int_{\Omega_{\veps \xi}}\kappa \left(\frac{x}{\epsilon},\frac{x}{\epsilon} + \xi \right) \nu(\xi) \cdot\frac{(m(x + \veps \xi)\times m(x))}{|\veps \xi|}\dd x\dd \xi.
\end{align*}

\begin{lemma}\label{lem:antisym_by_sym}
    Under hypotheses \ref{H1}--\ref{H4}, for every $m \in L^2(\Om; \R^3)$ there exists a constant \(C = C (a,\kappa, \nu, \rho) >0\) such that
    \begin{align*}
        -\frac{1}{2}\Fcal_\epsilon(m)
        -C\norm{m}_{L^2(\Omega)}^2
        \le\Hcal_\epsilon(m) 
        \le \frac{1}{2}\Fcal_\epsilon(m)
        +C\norm{m}_{L^2(\Omega)}^2.
    \end{align*}
\end{lemma}
\begin{proof}
We first observe that
\begin{align*}
    \Hcal_\epsilon(m) 
     &=\int_{\R^3 } \int_{\Omega_{\veps \xi}} 
    \kappa \left(\frac{x}{\epsilon},\frac{x}{\epsilon} + \xi \right)
    m(x+\veps\xi) \cdot\frac{( m(x)\times \nu(\xi))}{|\veps \xi|}\dd x\dd \xi\\
    & =\int_{\R^3 } \int_{\Omega_{\veps \xi}} 
    \kappa \left(\frac{x}{\epsilon},\frac{x}{\epsilon} + \xi \right) a^{1/2} \left(\frac{x}{\epsilon},\frac{x}{\epsilon} + \xi \right)
    \\
    &\qquad\qquad\qquad\qquad\qquad\rho(\xi)^{1/2}\frac{(m(x+\veps\xi)-m(x))}{|\veps \xi|}
    \cdot\left( \frac{m(x)}{a^{1/2} \left(\frac{x}{\epsilon},\frac{x}{\epsilon} + \xi \right)}\times \frac{\nu(\xi)}{\rho(\xi)^{1/2} }\right)\dd x\dd \xi.
\end{align*}
Hence, applying Young's inequality, we estimate (similar to \cite[Theorem~2.2, Step 2]{DDG24})
\begin{align}
\begin{split}
    |\Hcal_\epsilon(m)|
    &\le \frac{1}{2}\Fcal_\epsilon(m)
    +\frac{1}{2}\norm*{\frac{\kappa^2}{a}}_{L^\infty(\R^3\times \R^3)}
    \int_{\R^3}\int_\Om
    \frac{|m(x)\times \nu(\xi)|^2}{\rho(\xi)}  \dd x \dd \xi\\
    &\le \frac{1}{2}\Fcal_\epsilon(m)
    +\frac{1}{2}\frac{1}{a_0}\norm*{\kappa}^2_{L^\infty(\R^3\times \R^3)}\norm*{\frac{\nu}{\rho^{1/2}}}^{2}_{L^{2}(\R^3)}\norm{m}^2_{L^2(\Om)},
\end{split}
\end{align}
where we remind that \(a\ge a_0>0\) by \ref{H1}, and where we used that
\begin{align*}
    \int_{\R^3}\int_\Om\frac{|m(x)\times \nu(\xi)|^2}{\rho(\xi)}  \dd x \dd \xi
    \le \norm*{\frac{\nu}{\rho^{1/2}}}^{2}_{L^{2}(\R^3)}\norm{m}^2_{L^2(\Om)}
\end{align*}
by \ref{H4}. %, and that \(\norm{\nu}_{L^1(\R^3)}=1\) by \ref{H3}. 
The claim follows by %choosing \(\delta\coloneqq 1/\sqrt{2}\) and 
setting
\begin{equation*}
    C
    \coloneqq  \frac{1}{2}\frac{1}{a_0}\norm*{\kappa}^2_{L^\infty(\R^3\times \R^3)}
    \norm*{\frac{\nu}{\rho^{1/2}}}^{2}_{L^{2}(\R^3)}.\qedhere
\end{equation*}
\end{proof}

%\subsection{An alternative topology}
We prove next a density result with respect to the topology on the vector space of functions in \( L^2(\Om;L^2_{\#,0}(Q;\R^3))\) given by our nonlocal energies. %which is specifically tailored to the type of nonlocal energies that we are studying. 
This turned out to be useful for the \(\Gamma\)-convergence argument in Theorem~\ref{th:Gammalimit}. Adapting to our nonlocal setting, the proofs follow by standard techniques similar to the one employed for the fractional Sobolev spaces, see~e.g.,\cite[Theorem~6.6 and Theorem~6.62]{Leo23}. %, we provide the details since our setting is slightly different from those found in the literature.

Consider a kernel $\rho: \R^3 \to [0,+\infty]$ fulfilling condition \ref{H2}, and
%Given a measurable and non-negative kernel $\rho \in L^1(\R^3)$ fulfilling condition \eqref{eq:rhocoercive}, 
define the mapping \(\norm{\cdot}_\rho:L^2(\Om;L^2_{\#,0}(Q;\R^3))\to \R\), given by
\begin{equation}\label{eq:seminorm}
    \norm{v}_\rho
    \coloneqq \left(\int_{\Om}\int_{\R^3}\int_Q \rho(\xi)\frac{|v(x,z+\xi)-v(x,z)|^2}{|\xi|^2} {\dl z\dl \xi\dl x}\right)^\frac{1}{2}.
\end{equation}
For all \(v\in L^2(\Om;L^2_{\#,0}(Q;\R^3))\) we apply the nonlocal Poincar\'{e}-Wirtinger inequality in \cite[Proposition~4.2]{BeM14} to the function $v(x,\cdot)$ for a.e.~$x \in \Omega$, to find 
    \begin{align}
    \begin{split}\label{eq:nonlocal_poincare}
        \norm{v}^2_{L^2(\Om;L^2_{\#,0}(Q;\R^3))}
    &=\int_\Om\int_Q\abs*{v}^2\dd z\dd x \\
    &\leq C \int_\Om\int_Q\int_{Q\cap B_r(y)}\abs*{v(x,y)-v(x,z)}^2\dd z\dd y\dd x \\
    &\leq C \int_\Om\int_Q\int_{Q\cap B_r(y)}\rho(y-z)\frac{\abs*{v(x,y)-v(x,z)}^2}{\abs{y-z}^2}\dd z\dd y\dd x \\
    &\leq C \int_\Om\int_{\R^3}\int_Q\rho(\xi)\frac{\abs*{v(x,z+\xi)-v(x,z)}^2}{\abs{\xi}^2}\dd z\dd \xi\dd x\\
    &= C\norm{v}_\rho^2,
    \end{split}
    \end{align}
for some $C > 0$, where we have used \eqref{eq:rhocoercive} from \ref{H2} in the second inequality. We then have the following result.
\begin{proposition}\label{prop:hilbert_space}
 The map \(\norm{\cdot}_\rho\) in \eqref{eq:seminorm} defines a norm on the space 
\begin{equation*}
    H^\rho_{\#,0}\coloneqq \Big\{v\in L^2(\Om;L^2_{\#,0}(Q;\R^3)):\, 
    \norm{v}_\rho<+ \infty\Big\}.
\end{equation*}
Moreover, \((H^\rho_{\#,0},\norm{\cdot}_\rho) \) is a Hilbert space.
\end{proposition}
\begin{proof}
    It follows immediately from the definition in \eqref{eq:seminorm} that \(\norm{\cdot}_\rho\) is subadditive and positively homogeneous. The positive definiteness of \(\norm{\cdot}_\rho\) in \(H^\rho_{\#,0}\) is a direct consequence of \eqref{eq:nonlocal_poincare} and the positive definiteness of the \(L^2\)-norm.

    We turn now to the proof that \((H^\rho_{\#,0},\norm{\cdot}_\rho) \) is a Hilbert space. Due to the \(L^2\)-structure of \(\norm{\cdot}_\rho\) and the non-negativity of $\rho$, it is evident that the metric on \((H^\rho_{\#,0},\norm{\cdot}_\rho)\) is induced by a well-defined inner product.
    To show the completeness of \(H^\rho_{\#,0}\), let \(\seq{v}\subset H^\rho_{\#,0}\) be a Cauchy sequence. By \eqref{eq:nonlocal_poincare}, we find that $\seq{v}$ is also a Cauchy sequence in $L^2(\Om;L^2_{\#,0}(Q;\R^3))$, so that, up to a non-relabeled subsequence, there is a $v \in L^2(\Om;L^2_{\#,0}(Q;\R^3))$ such that $v_k \to v$ both in $L^2(\Om;L^2_{\#,0}(Q;\R^3))$ and pointwise a.e.~in $\Omega \times Q$. We apply Fatou's lemma to find that
    \begin{equation*}
        \lim_{k\to \infty}\norm{v-v_k}_\rho
        \le  \lim_{k\to \infty}\liminf_{l\to\infty} \norm{v_l-v_k}_\rho
        =0,
    \end{equation*}
    where the last equality follows from the fact that $\seq{v}$ is a Cauchy sequence.
    \end{proof}

We finally have the density result.

\begin{proposition}\label{prop:density}
    The inclusion \(C_c^\infty(\Om;C_{\#,0}^\infty( Q;\R^3))\subset H^\rho_{\#,0}\) is dense with respect to \(\norm{\cdot}_\rho\). 
    %In particular, for every \(v\in H^\rho_{\#,0}\) there exists a sequence \(\sequence{v}{k}\subset C_c^\infty(\Om;C_{\#,0}^\infty( Q;\R^3))\) such that \hidde{not needed anymore due to changes in proof recovery sequence}
    %\begin{equation*}
    %    \norm{v_k}_\rho\to \norm{v}_\rho
    %    \quad\text{and}\quad 
    %    \norm{v_k}_\rho\le \norm{v}_\rho
    %    \text{ for all } k\in \N.
    %\end{equation*}
\end{proposition}
\begin{proof}
    It is clear that \(C_c^\infty(\Om;C_{\#,0}^\infty( Q;\R^3))\subset H^\rho_{\#,0}\). Indeed, let $v \in C_c^\infty(\Om;C_{\#,0}^\infty( Q;\R^3))$, then we find that
     \begin{align*}
         \norm{v}^2_\rho &= \int_{\Om}\int_{\R^3}\int_{Q} \rho(\xi)\frac{|v(x,z+\xi)-v(x,z)|^2}{|\xi|^2} {\dl z\dl \xi\dl x} \\ & \leq \int_{\Om}\int_{\R^3}\int_{Q} \rho(\xi) \norm{\nabla v}^2_{L^{\infty}(\Omega\times\R^3)} \dl z \, \dl \xi\dl x \\ & \leq \abs{\Omega}\norm{\rho}_{L^1(\R^3)} \norm{\nabla v}^2_{L^\infty(\Om \times \R^3)} < + \infty.
     \end{align*}
    % By fundamental theorem of calculus, we have that \ROSS{(I don't know if the notation for the derivative wrt the second component is consistent with the rest of the paper)}
    % \begin{align*}
    %     \int_{\R^3} |v(x, z + \xi) - v(x,z)|^2  \dl z &\leq |\xi|^2 \int_{\R^3} \int_0^1 |\partial_2 v(x, z + t\xi)|^2 \dl t \dl z \\ &\leq |\xi|^2 \int_{\R^3} |\partial_2 v(x, y)|^2 \dl y.
    % \end{align*}
    % Thus, 
    % \begin{align*}
    %     \norm{v}^2_\rho &\leq \int_{\Om}\int_{\R^3}\int_{\R^3} \rho(\xi)\frac{|v(x,z+\xi)-v(x,z)|^2}{|\xi|^2} {\dl z\dl \xi\dl x} \\ & \leq \int_{\Om}\int_{\R^3}\int_{\R^3} \rho(\xi) |\partial_2 v(x, y)|^2 \dl y  \dl \xi\dl x \\ & \leq \norm{\rho}_{L^1(\R^3)} \norm{\nabla v}^2_{L^2(\Om \times \R^3)} < + \infty.
    % \end{align*}

    To prove the density, we proceed by mollification, with a similar construction as in Lemma~\ref{le:densitykernel}.
    Let \(\varphi\in C_c^\infty(\R^3\times \R^3)\), with $\varphi \geq 0$, \(\norm{\varphi}_{L^1(\R^3\times \R^3)}=1\) and \(\supp \varphi\subset B_1(0)\). We mollify simultaneously in the \(x\) and \(z\) arguments, and so we use the notation \(y=(y_1,y_2)\in \R^3\times\R^3\) for the mollifying variable, that is, for $k \in \N$ we define \(\varphi_k(y)\coloneqq k^{6}\varphi(ky)\). Moreover, we choose an associated sequence of cut-off functions \(\seq{\chi}\subset C_c^\infty(\Om;[0,1])\) such that \(\chi_k= 1\) on \(\Om_{3/k}\) and \(\chi_k=0\) on \(\Om\setminus \Om_{2/k}\) for all $k \in \N$. Then, for every \(v\in H^\rho_{\#,0}\) we define the mollified sequence 
    \begin{equation*}
        v_k(x,z)
        \coloneqq \chi_k(x)(\varphi_k*\bar{v})(x,z)
        = \chi_k(x)\int_{\R^6}\varphi_{k}(y)\bar{v}(x-y_1,z-y_2)\dl y \quad \text{for $k \in \N$},
    \end{equation*}
    where \(\bar{v}\) denotes the extension in \(x\) of \(v\) to the outside of \(\Om\) by zero. It follows that \(\seq{v}\subset C_c^\infty(\Om;C_{\#,0}^\infty( Q;\R^3))\). Since \(v\in L^2(\Om;L^2_\#(Q;\R^3))\), we have that \(v_k\to v\) in \(L^2(\Om;L^2_\#(Q;\R^3))\) and, possibly up to a non-relabeled subsequence, that \(v_k\to v\) pointwise a.e.~in \(\Om\times Q\) as $k \to \infty.$ 

    Additionally, we infer that \(\norm{v_k}_\rho\le \norm{v}_\rho\) for every \(k\in\N\). Indeed, by Jensen's inequality and Fubini's theorem, for a.e.~\(\xi \in \R^3\) we have that
    \begin{align}
        \begin{split}\label{eq:moll_lowers_en}
            &\int_{\Om}\int_Q |v_k(x,z+\xi)-v_k(x,z)|^2{\dl z\dl x}\\
            &\quad\le \int_{\R^6}\varphi_{k}(y)\int_{\Om_{2\veps_k}}\int_Q 
            |v(x-y_1,z-y_2+\xi)-v(x-y_1,z-y_2)|^2 {\dl z\dl x\dl y}\\
            &\quad = \int_{\R^6}\varphi_{k}(y)\int_{\Om_{2\veps_k}-y_1}\int_Q 
            |v(x,z+\xi)-v(x,z)|^2 {\dl z\dl x\dl y}\\
            &\quad \le \int_{\Om}\int_Q |v(x,z+\xi)-v(x,z)|^2{\dl z\dl x},
        \end{split}
    \end{align} where in the first inequality we treated \(\varphi_{k}\dl y\) as a probability measure on \(\R^6\), while in the equality we used that \(v_k\) is periodic in the second variable, and for the last inequality that \(\Om_{2/k}-y_1\subset \Om\) for every \(y\in \supp \varphi_{k}\subset B_{1/k}(0)\). Therefore, integrating also in $\xi \in \R^3$, we get
    \begin{align}\label{eq:bound_rho}
    \begin{split}
       \norm{v_k}^{2}_\rho & = \int_{\Om} \int_{\R^3}\int_Q \rho(\xi) \frac{|v_k(x,z+\xi)-v_k(x,z)|^2}{|\xi|^2} \dl z \dl \xi \dl x \\ &\leq \int_{\Om} \int_{\R^3}\int_Q \rho(\xi) \frac{|v(x,z+\xi)-v(x,z)|^2}{|\xi|^2} \dl z \dl \xi \dl x = \norm{v}^{2}_\rho .
       \end{split}
    \end{align}
    Using the bound \eqref{eq:bound_rho}, we infer that, up to a non-relabeled subsequence, $(v_k)_k$ converges weakly in $H^\rho_{\#,0}$ to some $\tilde{v} \in H^\rho_{\#,0}$, since $H^\rho_{\#,0}$ is a Hilbert space by Proposition \ref{prop:hilbert_space} (see, e.g., \cite[Theorem 3.18]{Bre11}).  Moreover, by \eqref{eq:nonlocal_poincare}  the topology on $H^\rho_{\#,0}$ is stronger than the one on $L^2(\Om;L^2_{\#,0}(Q;\R^3))$, and so $(v_k)_k$ also weakly converges to some $\tilde{v}$ in $L^2(\Om;L^2_{\#,0}(Q;\R^3))$. Since by construction of the mollified sequence \(v_k\to v\) in \(L^2(\Om;L^2_\#(Q;\R^3))\) as $k \to \infty$, 
    this implies that $\tilde{v}=v$ a.e.\ in $\Om$.
    Thus, the weak lower semicontinuity of the norm yields
    \[
    \limsup_{k \to \infty} \norm{v_k}_\rho \leq \norm{v}_\rho \leq \liminf_{k \to \infty} \norm{v_k}_\rho.
    \]
    This shows that \(\norm{v_k}_\rho\to \norm{v}_\rho\) as $k \to \infty$. Combining this with the weak convergence in $H^\rho_{\#,0}$, and using the Radon-Riesz property of uniformly convex Banach spaces (cf. \cite[Proposition~3.32]{Bre11}), we find that $v_k \to v$ in $H^\rho_{\#,0}$ as $k \to \infty$. This concludes the proof.
\end{proof}

\bibliographystyle{abbrv}
\bibliography{bibliography}
\end{document}